\documentclass[11pt]{amsart}
\usepackage{amsmath, amssymb, amsthm, xcolor, mathtools}
\usepackage[margin=1in]{geometry}
\usepackage{hyperref}
\usepackage{comment}
\usepackage{float}
\usepackage{enumitem}

\newtheorem{theorem}{Theorem}[section]
\newtheorem{lemma}[theorem]{Lemma}
\newtheorem{corollary}[theorem]{Corollary}
\newtheorem{proposition}[theorem]{Proposition}

\theoremstyle{definition}
\newtheorem{definition}[theorem]{Definition}

\theoremstyle{remark}
\newtheorem{remark}[theorem]{Remark}

\newtheorem{claim}{Claim}

\numberwithin{equation}{section}

\makeatletter

\renewcommand\section{%
  \@startsection{section}{1}%
    {\z@}%
    {3.9ex\@plus1ex\@minus.2ex}%
    {2.7ex\@plus.2ex}%
    {\normalfont\Large\scshape\centering}%
    }

\makeatletter
\renewcommand\subsection{%
  \@startsection{subsection}{2}%
    {\z@}%
    {3.25ex plus 1ex minus .2ex}%
    {1.5ex plus .2ex}%
    {\normalfont\large\scshape\centering}%
    }
\makeatother

\renewcommand{\leq}{\leqslant}
\renewcommand{\geq}{\geqslant}
\renewcommand\d{\,\mathrm{d}}
\newcommand{\E}{\mathbf{E}}
\newcommand{\Z}{\mathbb{Z}}

\newcommand{\N}{\mathbb{N}}

\newcommand{\eps}{\varepsilon}

\let\oldpmod\pmod
\renewcommand{\pmod}[1]{\hspace{-0.12cm}\oldpmod {#1}}
\newcommand{\qsiegel}{\widetilde{q}}

\usepackage[normalem]{ulem}

\usepackage[makeroom]{cancel}

\title{Sum-product patterns in the shifted primes}
\author[Florian~K.~Richter]{Florian Richter}
\address{\'{E}cole Polytechnique F\'{e}d\'{e}rale de Lausanne (EPFL), Lausanne, Switzerland}
\email{f.richter@epfl.ch}

\author[Joni Ter\"{a}v\"{a}inen]{Joni Ter\"{a}v\"{a}inen}
\address{Department of Pure Mathematics and Mathematical Statistics, University of Cambridge, Cambridge, UK}
\email{joni.p.teravainen@gmail.com}

\begin{document}

\begin{abstract}
We show that the set $\mathbb{P}-1$ of shifted primes contains infinitely many sum-product patterns of the form $\{x,x+y,xy\}$ with $x,y$ arbitrarily large distinct integers. More strongly, we can also show that, for any $k\geq 1$, the set $\mathbb{P}-1$ contains longer patterns of the form $\{x,x+y,\ldots, x+ky,xy\}$ with $x,y$ arbitrarily large distinct integers, a statement that contains the Green--Tao theorem as a special case.
\end{abstract}

\maketitle

\tableofcontents

\section{Introduction}

Let $\mathbb{P}$ denote the set of prime numbers.
In this paper we consider the existence of various sum-product patterns in the shifted primes $\mathbb{P}-1=\{p-1: p\in\mathbb{P}\}$. One of our results states the following.

\begin{theorem}\label{thm:sumprod}
The set $\mathbb{P}-1$ contains infinitely many patterns of the form $\{x,x+y,xy\}$ with $x,y$ arbitrarily large distinct integers.     
\end{theorem}

We cannot hope to find patterns of the form $\{x,y,x+y,xy\}$, or even $\{x,y,xy\}$, in $\mathbb{P}-1$ with $x,y$ arbitrarily large without at the same time solving a major open problem about the Liouville function. Recall that the Liouville function is defined by $\lambda(n)=(-1)^{\Omega(n)}$, where $\Omega(n)$ is the number of prime factors of $n$, counted with multiplicities.

\begin{proposition}
If the set $\mathbb{P}-1$ contains patterns of the form $\{x,y,xy\}$ with $x,y$ arbitrarily large integers, then there are infinitely many primes $p$ with
\begin{align*}
 \lambda(p-1)=+1.   
\end{align*}
\end{proposition}

\begin{proof}
If  $\{x,y,xy\}\subset \mathbb{P}-1$ with $x,y\geq N\geq 1$, then as $\lambda(xy)=\lambda(x)\lambda(y)$, at least one of $\lambda(x),\lambda(y),\lambda(xy)$ must be $+1$, giving a solution to $\lambda(p-1)=+1$ with $p\geq N+1$.   
\end{proof}

Theorem~\ref{thm:sumprod} is a special case of the following more general theorem that we establish.

\begin{theorem}[Main result]
\label{thm:linearforms}
Let $k\in \mathbb{N}$, $a_1,\ldots,a_k,b_1,\ldots,b_k\in\N$, and assume $a_ib_j-a_jb_i\neq 0$ whenever $i\neq j$. Let $L_1,\ldots, L_k\colon \mathbb{Z}^2\to \mathbb{Z}$ denote the linear forms
\[
L_j(x,y)=a_jx+b_jy,\quad j=1,\ldots,k.
\]
Then the set $\mathbb{P}-1$ contains infinitely many patterns of the form $\{x,xy,L_1(x,y),\ldots,L_k(x,y)\}$ with $x,y$ arbitrarily large distinct integers.     
\end{theorem}

If we choose $L_1(x,y)=x+y,\ldots, L_k(x,y)=x+ky$ in Theorem~\ref{thm:linearforms}, we obtain a corollary that contains both the Green--Tao theorem \cite{gt} and Theorem~\ref{thm:sumprod} as special cases.

\begin{corollary} For any $k\in \mathbb{N}$, the set $\mathbb{P}-1$ contains infinitely many patterns of the form $\{x,x+y,\ldots, x+ky,xy\}$ with $x,y$ arbitrarily large distinct integers.      
\end{corollary}

\subsection{Main asymptotic result}

For the remainder of the paper, let us fix $k,a_1,\ldots,a_k,b_1,\ldots,b_k\in\N$ and $L_1,\ldots,L_k\colon\Z^2\to\Z$ as in the statement of Theorem~\ref{thm:linearforms}.
The following two quantities depend only on $k$ and the linear forms and are therefore also fixed throughout the paper.

Throughout, let 
$$0<\delta\leq1/100$$
be sufficiently small in terms of $k$.
\begin{definition}[Definition of $\ell$ and $\Delta$]
\label{def_Delta}
Let
\begin{equation}
\label{def_ck_and_ell}
\ell=\lceil 30k/\delta\rceil,
\end{equation}
and let $\Delta= B!$, where $B$ is the maximum of 
$$\{k+3\}\cup\{|a_i|\colon 1\leq i\leq k\}\cup \{|b_i|\colon 1\leq i\leq k\}\cup \{|a_ib_j-a_jb_i|\colon 1\leq i<j\leq k\}.
$$ 
\end{definition} 

For notational convenience, define
\begin{align*}
\psi_{-1}(x)&=\Delta x+1;\\
\psi_0(x)&=\Delta^2x+1;\\
\psi_{j}(x,y)&= \Delta a_j x +\Delta b_jy+1=L_j(\Delta x,\Delta y)+1,\qquad 1\leq j\leq k.
\end{align*}
We will deduce Theorem~\ref{thm:linearforms} from the following asymptotic estimate.

\begin{theorem}[Main asymptotic result]
\label{thm_main_asymptotic}
For all sufficiently large $N$ there exist primes $p_1,\ldots,p_\ell \ll (\log N)(\log\log N)$ such that if
\begin{align}
\label{eq:w1choice}
W_1&=\prod_{p\leq w_1}p,\qquad \textnormal{where}~~w_1=\exp((\log N)^{1/4}),
\\
\label{eq:Mchoice}
M&=\lfloor N^{3/4}\rfloor,
\\
\label{eq:Dbound}
\mathcal{D}&=p_1^{D}\cdots p_{\ell}^{D},\qquad \textnormal{where}~~D=\lfloor (\log N)^{\delta}\rfloor,
\end{align}
then
\begin{align}
\label{eqn_main_asymptotic}
\sum_{d\mid \mathcal{D}}\sum_{\substack{m\leq M\\ (m,W_1)=1}}\sum_{n\leq N}\Lambda(\psi_{-1}(dm))\Lambda(\psi_0(\mathcal{D}mn))\prod_{j=1}^k \Lambda\big(\psi_j\big(dm,\tfrac{\mathcal{D}}{d}n\big)\big)\gg \frac{\tau(\mathcal{D})\varphi(W_1)M N}{W_1}.
\end{align}
\end{theorem}

\begin{remark}
The implied constant in the notation $\gg$ in \eqref{eqn_main_asymptotic} may depend on $k$ and on the coefficients $a_1,\ldots,a_k,b_1,\ldots,b_k$ of the linear forms, but is otherwise universal. Since we think of these parameters as fixed, we suppress this dependence from the notation. We keep this convention throughout the rest of the paper, and unless explicitly stated otherwise, all implied constants below may depend on $k$ and $a_1,\ldots,a_k,b_1,\ldots,b_k$.
\end{remark}

The upper bound sieve \cite[Fundamental Lemma~6.3]{iw-kow} gives
\[
\sum_{d\mid \mathcal{D}}\sum_{\substack{m\leq M\\ (m,W_1)=1}} \sum_{n\leq N} 1 \ll \frac{\tau(\mathcal{D})\varphi(W_1)M N}{W_1},
\]
and by crudely bounding the contribution of higher prime powers on the left-hand side of \eqref{eqn_main_asymptotic} by $\ll N^{3/2+o(1)}$, we see that Theorem~\ref{thm_main_asymptotic} provides a lower density estimate on the number of triples $(d,n,m)$ for which
\[
1_{\mathbb{P}}(\Delta dm+1)1_{\mathbb{P}}(\Delta^2\mathcal{D}mn+1)\prod_{j=1}^k 1_{\mathbb{P}}\big(L_j\big(\Delta dm,\tfrac{\Delta\mathcal{D}}{d}n\big)+1\big)>0.
\]
Taking $x=\Delta dm$ and $y=\frac{\Delta \mathcal{D}n}{d}$,
it follows that $\mathbb{P}-1$ contains $\{x,xy,L_1(x,y),\ldots,L_k(x,y)\}$ with $x,y$ arbitrarily large distinct integers.
This shows that Theorem~\ref{thm_main_asymptotic} implies
Theorem~\ref{thm:linearforms}.

\subsection{On the choice of $p_1,\ldots,p_\ell$}

The choice of $p_1,\ldots,p_\ell$ in the statement of Theorem~\ref{thm_main_asymptotic} is somewhat delicate and depends on whether Siegel zeros occur at certain levels associated to the scale $N$.

\begin{definition}[Siegel zero]
Let $\delta_0=1/1000$. Given $w\geq 2$, we say that a real number $\widetilde{\beta}\in(0,1)$ is a \emph{Siegel zero of level $w$} if there is a primitive Dirichlet character $\widetilde{\chi}\pmod{\qsiegel}$ with $\qsiegel\leq w$ for which
\[
L(\widetilde{\beta},\widetilde{\chi})=0
\quad\text{and}\quad
\widetilde{\beta}\ge 1-\frac{\delta_0}{\log w}.
\]
In this situation, the character $\widetilde{\chi}$ is called a \emph{Siegel character of level $w$}.
\end{definition}

From the Landau--Page theorem (see for example~\cite{pintz}), we know that if a Siegel zero of level $w$ exists, then it is unique and simple, and the corresponding Siegel character $\widetilde{\chi}$ is a unique non-principal real character. 
Moreover, we note for later use that the conductor $\qsiegel$ has squarefree odd part and $2$-part equal to either $1$, $4$, or $8$ (since there are no primitive real characters modulo $2^j$ for $j\geq 4$), and satisfies the lower bound
\begin{align}\label{eq:Siegelbound}
(\log w)^A\ll_A \qsiegel.
\end{align}

\begin{definition}[Definition of $\widetilde{\beta}_i$, $\widetilde{\chi}_i$, and $\qsiegel_i$]
\label{def_siegel_paratemers}
Let $w_1$ and $W_1$ be as in \eqref{eq:w1choice} and define
\begin{equation}
\label{eq:w2choice}
w_2=\exp(\frac{1}{2}(\log N)^{1/10}),
\qquad\text{and}\qquad
W_2=\prod_{p\leq w_2}p.
\end{equation}
For the remainder of the paper, for $i\in\{1,2\}$, let $\widetilde{\beta}_i$ and $\widetilde{\chi}_i \pmod{\qsiegel_i}$ denote the Siegel zero and the corresponding Siegel character of level $w_i$, respectively, when they exist. If no such zero exists, we set
\[
\widetilde{\beta}_i=1,\qquad \widetilde{\chi}_i\equiv 0,\qquad\text{and}\qquad \qsiegel_i=1.
\]
\end{definition}
Define
\begin{equation}
\label{def:Qtilde}
Q=\prod_{p\mid \operatorname{lcm}(\qsiegel_1,\qsiegel_2)}p
\end{equation}
and note that 
\begin{equation}
\label{eqn:Qbound}
Q\leq w_1w_2\leq \exp(2(\log N)^{1/4})
\end{equation}

We now have all the components needed to specify our choice of $p_1,\ldots,p_\ell$ at scale $N$.

\begin{definition}[Definition of $p_1,\ldots,p_\ell$]
\label{def_choice_of_primes}
Let $p_1<p_2<\ldots< p_\ell$ be the smallest $\ell$ primes that don't divide $Q$.
\end{definition}

Note that if there are no Siegel zeros at levels $w_1$ and $w_2$ then $Q=1$ and one can simply take $p_1,\ldots,p_\ell$ to be the smallest $\ell$ primes.
But even if not, $Q$ has at most $\ll \log N$ prime factors due to \eqref{def:Qtilde}, which implies that $p_{\ell}\ll (\log N)(\log \log N)$ as required in the hypothesis of Theorem~\ref{thm_main_asymptotic}. This also means that
\begin{align}
\label{eqn_bound_for_cD}
\mathcal{D}=\exp((\log N)^{\delta+o(1)}),    
\end{align}
which is an estimate that we will use several times throughout the proof.

\subsection*{Acknowledgments}
FKR was supported by the Swiss National Science Foundation grant TMSGI2-211214. JT was supported by funding from the European Union's Horizon Europe research and innovation programme under ERC grant agreement no. 101162746.

\subsection*{AI disclosure statement}
The authors used AI tools solely for language editing to improve the wording and readability of the manuscript. All mathematical ideas and proofs were developed by the authors, who assume responsibility for all content.

\section{Outline of the proof}

The proof of Theorem~\ref{thm_main_asymptotic} naturally splits into 4 steps, which we outline now.

\subsection{Step 1: Inner Siegel model approximation using Gowers norm estimates}

Given a positive integer $R$, define
\[
\Lambda_{R}(n)=\prod_{p\mid R}\Lambda_p(n),
\]
with the local von Mangoldt function $\Lambda_p$ defined as 
\begin{align}
\label{eq:lambdap}
\Lambda_p(n)\coloneqq \frac{p}{p-1}1_{n\not \equiv 0\pmod p}.    
\end{align}

We now introduce the Siegel model, which will play a central role in our proof.

\begin{definition}[The Siegel model]\label{def:Siegelmodel}
Given $w\in[2,\infty)$ and $W=\prod_{p\leq w}p$, we define the \emph{Siegel model at level $w$} as
\[
\widetilde{\Lambda}_{W}(n)
\coloneqq
\Lambda_{W}(n)\bigl(1-\widetilde{\chi}(n)n^{\widetilde{\beta}-1}\bigr),
\qquad n\in\N,
\]
where, if a Siegel zero exists at level $w$, $\widetilde{\beta}$ denotes this zero and $\widetilde{\chi}$ the corresponding Siegel character, and if no Siegel zero exists at that level, we take $\widetilde{\beta}=1$ and $\widetilde{\chi}\equiv 0$.
\end{definition}

The first step in the proof of Theorem~\ref{thm_main_asymptotic} is to replace the factors of $\Lambda$ involving the affine-linear forms $\psi_1,\ldots,\psi_k$ by a Siegel model. For the moment, we leave the two remaining factors involving $\psi_{-1}$ and $\psi_0$ unchanged, as these will be handled in subsequent steps. In this first replacement, we specifically use the Siegel model at level $w_2$, namely
\[
\widetilde{\Lambda}_{W_2}(n)
=
\Lambda_{W_2}(n)\bigl(1-\widetilde{\chi}_2(n)n^{\widetilde{\beta}_2-1}\bigr),
\]
where $w_2$, $\widetilde{\beta}_2$, and $\widetilde{\chi}_2$ are as in Definition~\ref{def_siegel_paratemers}.
Thus, we define
\[
S=\sum_{d\mid \mathcal{D}}\sum_{\substack{m\leq M\\ (m,W_1)=1}}\sum_{n\leq N}\Lambda(\psi_{-1}(dm))\Lambda(\psi_0(\mathcal{D}mn))\prod_{j=1}^k \Lambda\big(\psi_j\big(dm,\tfrac{\mathcal{D}}{d}n\big)\big)
\]
and
\[
S_{\mathrm{main}}^{(1)}=
\sum_{d\mid \mathcal{D}}\sum_{\substack{m\leq M\\ (m,W_1)=1}}\sum_{n\leq N}\Lambda(\psi_{-1}(dm))\Lambda(\psi_0(\mathcal{D}mn))\prod_{j=1}^k \widetilde{\Lambda}_{W_2}\big(\psi_j\big(dm,\tfrac{\mathcal{D}}{d}n\big)\big).
\]

\begin{proposition}[First main-term estimate]\label{prop:main}
We have
\begin{align*}
|S-S_{\mathrm{main}}^{(1)}| \ll \frac{\tau(\mathcal{D})MN}{(\log N)^{10k}}.   
\end{align*}    
\end{proposition}

The main input in the proof of Proposition~\ref{prop:main} is a local Gowers norm estimate for the Siegel-corrected von Mangoldt function $\Lambda -\widetilde{\Lambda}_{W_2}$.
We expand each factor
$\Lambda=\widetilde{\Lambda}_{W_2}
+(\Lambda-\widetilde{\Lambda}_{W_2})$ in $S$. For every binary vector $\mathbf{t}=(t_1,\ldots,t_k)\in \{0,1\}^k$, consider the sum
\[
S_{\mathrm{G},\mathbf{t}}= \sum_{d\mid \mathcal{D}}\sum_{\substack{m\leq M\\ (m,W_1)=1}}\sum_{n\leq N}\Lambda(\psi_{-1}(dm))\Lambda(\psi_0(\mathcal{D}mn))
\prod_{j=1}^k F_{t_j}\big(\psi_j\big(dm,\tfrac{\mathcal{D}}{d}n\big)\big),
\]
where
\[
F_{0}(n)=\widetilde{\Lambda}_{W_2}(n)\qquad\text{and}\qquad
F_{1}(n)=\big(\Lambda-\widetilde{\Lambda}_{W_2}\big)(n).
\]
Letting $\mathbf{0}=(0,\ldots,0)$, it is easy to see that 
\[
S_{\mathrm{main}}^{(1)}=S_{\mathrm{G},\mathbf{0}}\qquad\text{and}\qquad S= \sum_{ \mathbf{t}\in\{0,1\}^k} S_{\mathrm{G},\mathbf{t}}.
\]
Therefore, Proposition~\ref{prop:main} immediately follows from the following.

\begin{proposition}[Error-term estimate -- Gowers norm bound]
\label{prop_error_estimate_gowers_norm_bound}
For $\mathbf{0}\neq \mathbf{t}\in\{0,1\}^k$ we have
\begin{align}
\label{error:12-12-2}
|S_{\mathrm{G},\mathbf{t}}|
&\ll  \frac{\tau(\mathcal{D})MN}{(\log N)^{10k}}.
\end{align}
\end{proposition}

The proof of Proposition~\ref{prop_error_estimate_gowers_norm_bound} is found in Section~\ref{sec_proof_of_lem_error_estimate_gowers_norm_bound}.

\subsection{Step 2: Outer Siegel model approximation using Bombieri--Vinogradov estimate}

The next step in the proof is to replace $\Lambda(\psi_0(\mathcal{D}mn))$ by $\tfrac{m}{\varphi(m)}\widetilde{\Lambda}_{W_1}(\psi_0(\mathcal{D}mn))$, where $\widetilde{\Lambda}_{W_1}$ is the Siegel model at level $w_1$.
Consider the new main-term
\[
S_{\mathrm{main}}^{(2)}=
\sum_{d\mid \mathcal{D}}\sum_{\substack{m\leq M\\ (m,W_1)=1}}\sum_{n\leq N} \tfrac{m}{\varphi(m)} \Lambda(\psi_{-1}(dm))\widetilde{\Lambda}_{W_1}(\psi_0(\mathcal{D}mn))\prod_{j=1}^k \widetilde{\Lambda}_{W_2}\big(\psi_j\big(dm,\tfrac{\mathcal{D}}{d}n\big)\big).
\]

\begin{proposition}[Second main-term estimate]\label{prop:maintwo}
We have
\begin{align*}
|S_{\mathrm{main}}^{(1)}-S_{\mathrm{main}}^{(2)}| \ll \tau(\mathcal{D})MN\exp(-(\log N)^{1/20}).
\end{align*}    
\end{proposition}

The proof relies on a variant of the Bombieri--Vinogradov theorem, Lemma~\ref{le:BV}.
Consider the remainder term
\begin{align*}
S_{\mathrm{BV}}&=
\sum_{d\mid \mathcal{D}}\sum_{\substack{m\leq M\\ (m,W_1)=1}}\sum_{n\leq N} \Lambda(\psi_{-1}(dm)) \big(\Lambda-\tfrac{m}{\varphi(m)}\widetilde{\Lambda}_{W_1}\big)(\psi_0(\mathcal{D}mn))\prod_{j=1}^k \widetilde{\Lambda}_{W_2}\big(\psi_j\big(dm,\tfrac{\mathcal{D}}{d}n\big)\big)
\end{align*}
so that
\[
S_{\mathrm{main}}^{(1)}=S_{\mathrm{main}}^{(2)} + S_{\mathrm{BV}}.
\]
Hence Proposition~\ref{prop:maintwo} is immediately seen to be equivalent to the next result.

\begin{proposition}[Error-term estimate -- Bombieri--Vinogradov bound]
\label{prop_error_estimate_bombieri-vinogradov_bound}
We have
\begin{align*}
|S_{\mathrm{BV}}|
&\ll \tau(\mathcal{D})MN\exp(-(\log N)^{1/20}).
\end{align*}
\end{proposition}

The proof of Proposition~\ref{prop_error_estimate_bombieri-vinogradov_bound} is carried out in Section~\ref{sec_proof_of_lem_error_estimate_BV_bound}.

\subsection{Step 3: Removal of Siegel correction and singular series estimate}

In this next step, we replace the sum in $n$ with a singular series. Consider
\[
S_{\mathrm{main}}^{(3)}= N
\sum_{d\mid \mathcal{D}}\sum_{\substack{m\leq M\\ (m,W_1)=1}} \tfrac{m}{\varphi(m)} \Lambda(\psi_{-1}(dm))\mathfrak{S}(m,d),
\]
where $\mathfrak{S}(m,d)$ denotes the product
\[
\mathfrak{S}(m,d)=\prod_{p\leq w_1}\beta_p(m,d)
\]
with
\begin{equation}
\label{def_beta_dm}
\beta_p(m,d)
=
\begin{cases}
\frac{1}{p(1-\frac{1}{p})^{k+1}}\# \big\{x\in \mathbb{Z}/p\mathbb{Z}:
p\nmid \psi_0(\mathcal{D}mx)\prod_{j=1}^k \psi_j\big(dm,\tfrac{\mathcal{D}}{d}x\big)\big\},&\text{if}~p\leq w_2\\
\frac{1}{p(1-\frac{1}{p})}\# \big\{x\in \mathbb{Z}/p\mathbb{Z}:
p\nmid \psi_0(\mathcal{D}mx) \big\},&\text{if}~w_2< p\leq w_1
\end{cases}
\end{equation}

The goal is to show the following proposition.
\begin{proposition}[Third main-term estimate]\label{prop:mainthree}
For every $A\geq 1$ we have
\begin{align*}
|S_{\mathrm{main}}^{(2)}-S_{\mathrm{main}}^{(3)}| \ll_A \frac{\tau(\mathcal{D})MN}{(\log N)^A}.
\end{align*}
\end{proposition}

The proof of Proposition~\ref{prop:mainthree} is done in two steps. 
Consider the variant of $S_{\mathrm{main}}^{(2)}$
where all appearances of the Siegel model are replaced with the Cramer model:
\[
S_{\mathrm{main}}^{(2.5)}=
\sum_{d\mid \mathcal{D}}\sum_{\substack{m\leq M\\ (m,W_1)=1}}\sum_{n\leq N} \tfrac{m}{\varphi(m)} \Lambda(\psi_{-1}(dm))\Lambda_{W_1}(\psi_0(\mathcal{D}mn))\prod_{j=1}^k \Lambda_{W_2}\big(\psi_j\big(dm,\tfrac{\mathcal{D}}{d}n\big)\big).
\]
Clearly, Proposition~\ref{prop:mainthree} follows readily by combining the next two estimates.

\begin{proposition}[Error-term estimate -- Sieve bound]\label{prop:sieve-bound}
We have
\begin{align*}
|S_{\mathrm{main}}^{(3)}-S_{\mathrm{main}}^{(2.5)}| \ll \tau(\mathcal{D})MN\exp\bigl(-(\log N)^{1/2}\bigr).
\end{align*}
\end{proposition}

\begin{proposition}[Error-term estimate -- Weil bound]\label{prop:weil-bound}
For every $A\geq 1$ we have
\begin{align*}
|S_{\mathrm{main}}^{(2)}-S_{\mathrm{main}}^{(2.5)}| \ll_A \frac{\tau(\mathcal{D})MN}{(\log N)^A}.
\end{align*}
\end{proposition}

The main technical ingredient in the proof of Proposition~\ref{prop:sieve-bound} is a corollary of the fundamental lemma of sieve theory, Lemma~\ref{lem_singular_series_estimate}. The proof of Proposition~\ref{prop:weil-bound} relies on the same lemma,
but also crucially on classical Weil bounds for character sums. Both Propositions are proved in Section~\ref{sec_proof_of_sieve_and_weil_bound}.

\subsection{Step 4: Calculating the simplified main term}

The last step in the proof of Theorem~\ref{thm_main_asymptotic} is to show the following:

\begin{proposition}[Last main-term estimate]
\label{prop:mainfour}
We have
\[
S_{\mathrm{main}}^{(3)}\gg \frac{\tau(\mathcal{D})\varphi(W_1)MN}{W_1}.
\]
\end{proposition}

Observe that Propositions~\ref{prop:main}, \ref{prop:maintwo}, \ref{prop:mainthree}, and \ref{prop:mainfour} do indeed combine to yield a complete proof of Theorem~\ref{thm_main_asymptotic}.

The following proposition divides Proposition~\ref{prop:mainfour} into three smaller parts.

\begin{proposition}
\label{prop_last}
\begin{enumerate}
[label=(\roman{enumi}),ref=(\roman{enumi}),leftmargin=*]
\item\label{itm_last_i}
For all $m\leq M$ with $(m,W_1)=1$ and all $d\mid \mathcal{D}$ with $\textnormal{rad}(\mathcal{D})\mid d$, we have
\begin{align*}
\mathfrak{S}(m,d)\gg 1;    
\end{align*}
\item\label{itm_last_ii}
For all $m\leq M$ with $(m,W_1)=1$ we have
\[
0\leq \frac{m}{\varphi(m)}-1\ll \exp\bigl(-(\log N)^{1/5}\bigr);
\]
\item\label{itm_last_iii}
For every $d\mid\mathcal{D}$ we have the lower bound
\begin{equation}
\label{eqn_itm_last_iii}
\sum_{\substack{m\leq M\\ (m,W_1)=1}}  \Lambda(\psi_{-1}(dm)) \gg \frac{\varphi(W_1)M}{W_1}.
\end{equation}
\end{enumerate}
\end{proposition}

The proof of Proposition~\ref{prop_last} is in Section~\ref{sec_proof_last_estimate}.

\section{Preparatory results}

In this section, we prepare for the proofs of Propositions~\ref{prop_error_estimate_gowers_norm_bound},~\ref{prop:mainthree}, and~\ref{prop_error_estimate_bombieri-vinogradov_bound}. More precisely, we collect auxiliary lemmas used in the proof of these propositions. Some of the lemmas are quoted from the literature, while others are proved here.

Throughout, let $w\geq 2$ be real and set $W=\prod_{p\leq w}p$. Here $w$ should be regarded as a placeholder for either of the two parameters $w_1$ and $w_2$ introduced in \eqref{eq:w1choice} and \eqref{eq:w2choice}.
We also write $\widetilde{\beta}$ and $\widetilde{\chi}\pmod{\qsiegel}$ for the Siegel zero and the corresponding Siegel character of level $w$, whenever they exist, or let $\widetilde{\beta}=1$, $\widetilde{\chi}\equiv 0$, and $\qsiegel=1$ if no Siegel zero at this level exists, and consider the Siegel model
$\widetilde{\Lambda}_{W}(n)=
\Lambda_{W}(n)\bigl(1-\widetilde{\chi}(n)n^{\widetilde{\beta}-1}\bigr)$.
Our goal is to establish estimates for the Siegel model $\widetilde{\Lambda}_{W}$ and for the error term $\Lambda-\widetilde{\Lambda}_{W}$, including sieve estimates, Gowers norm estimates, and asymptotic formulas for averages along arithmetic progressions.

\subsection{An $L^1$ Type-I approximation of $\Lambda_R$}

We begin with a lemma showing that  $\Lambda_{R}$ is very close in $L^1$-norm to a type~I sum.
\begin{lemma}\label{le:Lambdareplace}
Let $w\geq 2$, $W=\prod_{p\leq w} p$, and $s\geq 100\log \log(2w)$.
Let $R\mid W$ and define
\begin{align}
\label{def_typeI-sum}
\Lambda_{R}^{(I)}(n)\coloneqq \frac{R}{\varphi(R)}\sum_{\substack{r\mid R \\ r\leq w^s}} \mu(r) 1_{r\mid n}.      
\end{align}   
Then for any $a,q\in\N$ with $(a,q)\leq e^{s/4}$, and any $X\geq \max\{w^{10s},q^{10/9}\}$, we have 
\begin{align*}
\sum_{\substack{n\leq X\\ n\equiv a \pmod q}}|\Lambda_{R}(n)-\Lambda_R^{(I)}(n)|\ll \frac{X}{q} e^{-s/12}.
\end{align*}    
\end{lemma}

\begin{proof}
By the triangle inequality it suffices to show that
\begin{align*}
\sum_{\substack{n\leq X\\ n\equiv a \pmod q}}\,\frac{R}{\varphi(R)}\sum_{\substack{r\mid R\\ r>w^s}}1_{r\mid n}\ll \frac{X}{q}e^{-s/12}.    
\end{align*}
Let $g=(a,q)$. Then any $n\equiv a\pmod q$ can be written as $n=gn'$ with $n'\equiv a/g\pmod{q/g}$. For each $r\mid(n,R)$, let $g'=(r,g)$ and $r'=r/g'$. Since $(r',g/g')=1$ and $r'\mid(g/g')n'$, we have $r'\mid n'$. Moreover, $g'\mid(g,R)$, so for each $r'$ there are at most $\tau((g,R))$ possible values of $g'$. Finally, $g'\leq g\leq e^{s/4}\leq w^{s/2}$, so $r>w^s$ implies $r'>w^{s/2}$. Since $R/\varphi(R)\ll\log w\ll e^{s/100}$, it suffices to show that
\begin{align*}
\sum_{\substack{n'\leq X/g\\ n'\equiv a/g \pmod{q/g}}}\,\sum_{\substack{r'\mid R\\ r'>w^{s/2}}}1_{r'\mid n'}\cdot \tau((g,R))\ll \frac{X}{q}e^{-s/10}.    
\end{align*}
Since $r'\mid R$ and $r'>w^{s/2}$ imply $\omega((r',R))>s/2$, for any $n'\in \mathbb{N}$ we have
\begin{align*}
\sum_{\substack{r'\mid R\\ r'>w^{s/2}}}1_{r'\mid n'}\ll 2^{\omega((n',R))}1_{\omega((n',R))>s/2}\leq 2^{-s/2}4^{\omega((n',R))}.    
\end{align*}
Hence, using $\tau((g,R))\leq \tau(g)\ll g^{1/100}$, it suffices to show that
\begin{align*}
\sum_{\substack{n'\leq X/g\\ n'\equiv a/g \pmod{q/g}}}\, 4^{\omega((n',R))}\ll\frac{X}{q}2^{s/2}e^{-s/8}. 
\end{align*}
Since $(a/g,q/g)=1$, $R\mid W$ and $q/g\leq (X/g)^{99/100}$, Shiu's bound~\cite[Theorem 1]{shiu}  then shows that the left-hand side is
\begin{align*}
\ll \frac{X/g}{\varphi(q/g)}\frac{1}{\log X}\exp\left(\sum_{\substack{p\leq X\\p\nmid q/g}}\frac{4^{\omega((p,R))}}{p}\right)\ll\frac{X/g}{q/g}(\log (2w))^3\ll \frac{X}{q}e^{s/10},    
\end{align*}
and the claim follows.
\end{proof}

\subsection{Local Gowers norm bounds for the Siegel-corrected von Mangoldt function}
\label{prelims_uniformity_norms}

For a prime $p$, a function $f\colon\Z/p\Z\to\mathbb C$, and
$s\in\N$, we define the degree $s$ Gowers norm of $f$ by
\[
\|f\|_{U^s(\Z/p\Z)}
\coloneqq
\left(
\mathbb E_{x,h_1,\ldots,h_s\in\Z/p\Z}
\prod_{\omega\in\{0,1\}^s}
\mathcal C^{|\omega|}
f\bigl(x+\omega\cdot(h_1,\ldots,h_s)\bigr)
\right)^{1/2^s},
\]
where $\mathcal C(z)=\overline z$, $|\omega|=\omega_1+\cdots+\omega_s$,
and $\omega\cdot(h_1,\ldots,h_s)=\omega_1h_1+\cdots+\omega_sh_s$. The same definition applies to $\Z/q\Z$ for every integer $q\geq1$.
For a finitely supported function $f\colon\Z\to\mathbb C$, write
\[
\|f\|_{\widetilde U^s(\Z)}^{2^s}
\coloneqq
\sum_{x,h_1,\ldots,h_s\in\Z}
\prod_{\omega\in\{0,1\}^s}
\mathcal C^{|\omega|}
f\bigl(x+\omega\cdot(h_1,\ldots,h_s)\bigr).
\]

For a discrete interval $I=x+[N]=\{x+1,x+2,\ldots,x+N\}$, choose a prime $p$ satisfying $p>2^{s+1}N$.
Given a function $f\colon\N\to\mathbb{C}$, define $f_{I,p}\colon\Z/p\Z\to\mathbb C$ by
\[
f_{I,p}(r)
=
\begin{cases}
f(x+r), & \text{if}~1\leq r\leq N\\
0, & \text{otherwise}.
\end{cases}
\]
We define
\[
\|f\|_{U^s(I)}
\coloneqq
\frac{\|f_{I,p}\|_{U^s(\Z/p\Z)}}
{\|1_{[N]}\|_{U^s(\Z/p\Z)}}.
\]
The condition $p>2^{s+1}N$ ensures that the additive cubes contributing
to these Gowers norms do not wrap around modulo $p$. In particular,
the above definition is independent of the choice of $p$.
When $I=[N]$, we write $\|f\|_{U^s[N]}$ instead of $\|f\|_{U^s(I)}$. For real $X\geq0$ and $H\geq1$, we interpret $\|f\|_{U^s(X,X+H]}$ as the norm on $(X,X+H]\cap\mathbb N=\lfloor X\rfloor+[\lfloor X+H\rfloor - \lfloor X\rfloor]$.

The following result provides estimates on the Gowers norm of $\Lambda-\widetilde{\Lambda}_W$ over short intervals and plays a central role in the proof of our main result. It is established in a companion-paper to this paper, see~\cite{Teravainen-MSTT}. 

\begin{theorem}[Local Gowers norm bound for $\Lambda-\widetilde{\Lambda}_W$]
\label{lem_local_uniformity_norm_bound_for_siegel-corrected_vonMangoldt}
Fix an integer $s\geq2$ and $0<\eps<3/16$, and set $a_s=2^{-s-2}$.
There is $\gamma=\gamma(s,\eps)>0$ such that the following holds.
Let $X\geq3$, $H\geq1$, $w\geq2$, and define $W=\prod_{p\leq w}p$. If
\begin{equation*}
 X^{5/8+\eps}\leq H\leq X
\qquad
\text{and}
\qquad
w\leq\exp\bigl((\log X)^{1/10}\bigr),
\end{equation*}
then
\begin{equation*}
 \|\Lambda-\widetilde{\Lambda}_W\|_{U^s(X,X+H]}
 \ll_{s,\eps}w^{-a_s}+\exp\bigl(-(\log X)^\gamma\bigr).
\end{equation*}
\end{theorem}

\subsection{A Variant of Bombieri--Vinogradov with Siegel correction}

The following lemma provides a Bombieri--Vinogradov type estimate where we have an additional congruence, possibly to a moderately large modulus, but we work with $\Lambda-\widetilde{\Lambda}_W$.

\begin{lemma}[Bombieri--Vinogradov theorem with additional congruence]\label{le:BV}
There is an absolute constant $C_0\geq 1$ such that the following holds.
Let $N\geq 2$ and let $q\geq 1$ be an integer, and let $w\geq 2$ be real with $q\leq w\leq \exp((\log N)^{1/3})$. Also let $1\leq M'\leq N/q^5$. Then
\begin{align}\label{eq:BV}\begin{aligned}
\sum_{\substack{m\leq M'\\ (m,W)=1}}
~\max_{\substack{b\pmod {mq}\\ (b,m)=1}}
~\max_{1\leq y\leq M'N} 
&\left|
\sum_{\substack{n\leq y\\ n\equiv b\pmod {mq}}} 
\left(\Lambda(n)-\frac{m}{\varphi(m)}\widetilde{\Lambda}_W(n)\right)
\right|
\\
&\ll 
M'N\left(\left(\frac{M'}{N}\right)^{1/6}+w^{-1/3}\right)(\log N)^{C_0}.
\end{aligned}
\end{align}
\end{lemma}

\begin{proof}
Throughout the proof, $C_0$ is a large constant that may vary from line to line. Let
\begin{align*}
        X\coloneqq M'N,\qquad \theta\coloneqq \frac{M'}{N}.
\end{align*}

First consider residue classes with $(b,q)>1$. Since $q\leq w$,
every integer $n\equiv b\pmod {mq}$ is divisible by some prime $p\leq w$,
so $\widetilde{\Lambda}_W(n)=0$ on this progression and $\Lambda(n)$ can be nonzero only when $n$ is a power of $p$.
Thus
\begin{align*}
\sum_{\substack{n\leq y\\ n\equiv b\pmod {mq}}}
\left(\Lambda(n)-\frac{m}{\varphi(m)}\widetilde{\Lambda}_W(n)\right)
\ll (\log X)^2.
\end{align*}
After summing over $m\leq M'$, this is
\begin{align*}
        \ll M'(\log X)^2
        \ll M'Nw^{-1/3},
\end{align*}
say.
We may therefore assume from now on that $(b,q)=1$, and hence, since $(b,m)=1$, that $(b,mq)=1$.

Recall that
\[
\widetilde{\Lambda}_W(n)=\Lambda_{W}(n)\bigl(1-\widetilde{\chi}(n)n^{\widetilde{\beta}-1}\bigr)
\]
and define
\[
\widetilde{\Lambda}_W^{(I)}(n)\coloneqq \Lambda_{W}^{(I)}(n)\bigl(1-\widetilde{\chi}(n)n^{\widetilde{\beta}-1}\bigr),
\]
where $\Lambda_{W}^{(I)}(n)$ is as in \eqref{def_typeI-sum} with 
$$
s_0=1000\log w.
$$
Lemma~\ref{le:Lambdareplace}, applied at $X$, gives for $y\leq X$ the bound
\begin{align*}
\sum_{\substack{n\leq y\\ n\equiv b\pmod{mq}}}\bigl|\widetilde{\Lambda}_W(n)-\widetilde{\Lambda}_W^{(I)}(n)\bigr|
&\ll\sum_{\substack{n\leq X\\ n\equiv b\pmod{mq}}}\bigl|\Lambda_W(n)-\Lambda_W^{(I)}(n)\bigr|\\
&\ll \frac{X}{mq}e^{-s_0/12}.
\end{align*}
Multiplying by $m/\varphi(m)$ and summing over $m\leq M'$ gives
\[
\ll\frac{X}{q}e^{-s_0/12}\sum_{m\leq M'}\frac1{\varphi(m)}\ll Xw^{-1/3}(\log N)^{C_0}.
\]
Hence, it remains to prove~\eqref{eq:BV} with $\widetilde{\Lambda}_W$ replaced by
$\widetilde{\Lambda}_W^{(I)}$.

For positive integers $m$, define
\begin{align*}
F_m(n) \coloneqq
\Lambda(n)-\frac{m}{\varphi(m)}\widetilde{\Lambda}_W^{(I)}(n).
\end{align*}
Also, for a function $f\colon\mathbb N\to\mathbb C$, define
\begin{align*}
D_{q,b,f}(Y;m,a) \coloneqq \sum_{\substack{n\leq Y\\ n\equiv b\pmod q\\ n\equiv a\pmod m}}f(n)
- \frac{1}{\varphi(m)}
\sum_{\substack{n\leq Y\\ n\equiv b\pmod q\\ (n,m)=1}}f(n).
\end{align*}
The congruence $n\equiv b\pmod {mq}$ is equivalent to
$n\equiv b\pmod q$ and $n\equiv b\pmod m$, so
\begin{align}\label{eq:BV-split}
\sum_{\substack{n\leq Y\\ n\equiv b\pmod {mq}}}F_m(n)
=D_{q,b,F_m}(Y;m,b)
+\frac{1}{\varphi(m)}
\sum_{\substack{n\leq Y\\ n\equiv b\pmod q\\ (n,m)=1}}F_m(n).
\end{align}

We first estimate the second term in \eqref{eq:BV-split}. Since
$\Lambda$ is supported on prime powers and $(m,q)=1$, removing the
condition $(n,m)=1$ from the sum of $\Lambda$ over $n\equiv b\pmod q$
only adds powers of primes dividing $m$. Hence
\begin{align}\label{eq:Lambda-coprime-m}
\sum_{\substack{n\leq Y\\ n\equiv b\pmod q\\ (n,m)=1}}\Lambda(n)
=\sum_{\substack{n\leq Y\\ n\equiv b\pmod q}}\Lambda(n)
+O(\omega(m)(\log X)^2).
\end{align}
Suppose first that $Y>X/w^2$. Applying the Landau--Page theorem in the form of~\cite[Theorem~5.27]{iw-kow}, we see that for some small absolute constant $c>0$ we have
\begin{align}\label{eq:LP-q}
\sum_{\substack{n\leq Y\\ n\equiv b\pmod q}}\Lambda(n)
=\frac{Y}{\varphi(q)}-\frac{\chi_q(b)}{\varphi(q)}
\frac{Y^{\beta_q}}{\beta_q}
+O\left(Y\exp\left(-c\frac{\log Y}{\sqrt{\log Y}+\log q}\right)(\log (2q))^4\right),
\end{align}
where $\chi_q$ is a real character $\pmod q$, $\beta_q>1-\frac{c}{\log(2q)}$, and $L(\beta_q,\chi_q)=0$. If $\chi_q$ does not exist, we delete the term involving $\chi_q$ and $\beta_q$.

The error term in~\eqref{eq:LP-q} is $O(Xw^{-1})$ since $q\leq w$ and $\log Y\asymp \log X\asymp \log N$, so
\begin{align}\label{eq:LP-q2}
\sum_{\substack{n\leq Y\\ n\equiv b\pmod q}}\Lambda(n)
=\frac{Y}{\varphi(q)}-\frac{\chi_q(b)}{\varphi(q)}
\frac{Y^{\beta_q}}{\beta_q}
+O(Xw^{-1}),
\end{align}
The same holds also for $Y\leq X/w^2$, since in that case we can crudely use Chebyshev's bound to obtain the same conclusion. If $\chi_q$ exists and
\begin{align}\label{eq:non-siegel-zero}
\beta_q<1-\frac{\delta_0}{\log w},
\end{align}
then the exceptional term in~\eqref{eq:LP-q2} is also $O(Xw^{-1})$. Otherwise, if $\chi_q$ exists, $\beta_q$ is a Siegel zero of level $w$. By uniqueness, the primitive character inducing $\chi_q$ is $\widetilde{\chi}$. Its conductor is $\qsiegel$, so this occurs precisely when $\qsiegel\mid q$, and then $\beta_q=\widetilde{\beta}$. Hence
\begin{align}\label{eq:Lambda-conclusion}
\sum_{\substack{n\leq Y\\ n\equiv b\pmod q}}\Lambda(n)= \frac{Y}{\varphi(q)}-1_{\qsiegel\mid q}\frac{\widetilde{\chi}(b)}{\varphi(q)}
\frac{Y^{\widetilde{\beta}}}{\widetilde{\beta}}
+O(Xw^{-1}).
\end{align}

We next compute the corresponding model sum.  We first treat the untwisted part.  Expanding $\Lambda_W^{(I)}$ and then using inclusion-exclusion for the condition $(n,m)=1$, we obtain
\begin{align*}
&\sum_{\substack{n\leq Y\\ n\equiv b\pmod q\\ (n,m)=1}}
\Lambda_W^{(I)}(n) =\frac{W}{\varphi(W)}
\sum_{\substack{r\mid W\\ r\leq w^{s_0}}}\mu(r)
\sum_{d\mid m}\mu(d)
|\{n\leq Y\colon n\equiv b\pmod q,\ r\mid n,\ d\mid n\}|.
\end{align*}
Since $(b,q)=1$, the congruences $n\equiv b\pmod q$ and $r\mid n$ are incompatible unless $(r,q)=1$. Moreover, since
every prime divisor of $q$ divides $W$, since $r\mid W$, and since $d\mid m$ with $(m,W)=1$, we have $(q,d)=1$ and $(r,d)=1$. Hence, when
$(r,q)=1$, the three congruences
\begin{align*}
        n\equiv b\pmod q,\qquad r\mid n,\qquad d\mid n
\end{align*}
have pairwise coprime moduli and combine by the Chinese remainder theorem into a single residue class modulo $qrd$.  Thus
\begin{align*}
|\{n\leq Y\colon n\equiv b\pmod q,\ r\mid n,\ d\mid n\}|
=\frac{Y}{qrd}+O(1).
\end{align*}
It follows that
\begin{align*}
&\sum_{\substack{n\leq Y\\ n\equiv b\pmod q\\ (n,m)=1}}
\Lambda_W^{(I)}(n)=
\frac{W}{\varphi(W)}
\frac{Y}{q}
\sum_{\substack{r\mid W\\ r\leq w^{s_0}\\ (r,q)=1}}\frac{\mu(r)}{r}
\sum_{d\mid m}\frac{\mu(d)}{d}
+O\left(\frac{W}{\varphi(W)}\tau(m)w^{s_0}\right).
\end{align*}
Note that we have the tail estimate
\begin{align}\label{eq:smoothtail}
 \sum_{\substack{r\mid W\\ r>w^{s_0}}}\frac{1}{r}\ll e^{-s_0}\prod_{p\leq w}\left(1+p^{-1+1/(\log w)}\right)\ll e^{-s_0}(\log w)\ll e^{-s_0/2}.   
\end{align}
From this it follows that
\begin{align*}
\frac{W}{\varphi(W)}
\sum_{\substack{r\mid W\\ r\leq w^{s_0}\\ (r,q)=1}}\frac{\mu(r)}{r}
=\frac{W}{\varphi(W)}
\sum_{\substack{r\mid W\\  (r,q)=1}}\frac{\mu(r)}{r}+O(w^{-1})=\frac{q}{\varphi(q)}+O(w^{-1}),
\end{align*}
say. Therefore
\begin{align}\label{eq:model-untwisted-principal}
\sum_{\substack{n\leq Y\\ n\equiv b\pmod q\\ (n,m)=1}}
\Lambda_W^{(I)}(n)
=\frac{\varphi(m)}{m}\frac{Y}{\varphi(q)}
+O\left(Yw^{-1}+\tau(m)w^{s_0}\log w\right).
\end{align}

We now treat the twisted part. If
there is no Siegel zero, then $\widetilde{\chi}\equiv 0$ and there is
nothing to prove. Otherwise, expanding as above gives
\begin{align}\label{eq:lambda-twist-expand}
\sum_{\substack{n\leq Y\\ n\equiv b\pmod q\\ (n,m)=1}}
\Lambda_W^{(I)}(n)\widetilde{\chi}(n)n^{\widetilde{\beta}-1}=
\frac{W}{\varphi(W)}
\sum_{\substack{r\mid W\\ r\leq w^{s_0}}}\mu(r)
\sum_{d\mid m}\mu(d)
\sum_{\substack{n\leq Y\\ n\equiv b\pmod q\\ r\mid n\\ d\mid n}}
\widetilde{\chi}(n)n^{\widetilde{\beta}-1}.
\end{align}
We may again restrict to the terms with $(r,q)=(d,q)=1$ as the other terms vanish. The terms with $(r,\qsiegel)>1$ vanish. Indeed, if
$p\mid (r,\qsiegel)$, then $r\mid n$ implies $p\mid n$,
and hence $\widetilde{\chi}(n)=0$. We may therefore restrict to
$(r,\qsiegel)=1$. Also $(d,\qsiegel)=1$,
since $d\mid m$, $(m,W)=1$, and every prime divisor of
$\qsiegel$ divides $W$.

Write
\begin{align*}
 \widetilde{\chi}(n) =\sum_{h\pmod {\qsiegel}}
        \widetilde{\chi}(h)1_{n\equiv h\pmod {\qsiegel}}.
\end{align*}
Since $(rd,\qsiegel)=1$ and $(rd,q)=1$, for a given $h$ the congruences
\begin{align*}
        n\equiv b\pmod q,\qquad r\mid n,\qquad d\mid n,\qquad
        n\equiv h\pmod {\qsiegel}
\end{align*}
are compatible if and only if
\begin{align*}
        h\equiv b\pmod{(q,\qsiegel)}.
\end{align*}
When this holds, since $(rd,\qsiegel)=1$ and also
$(rd,q)=1$, the Chinese remainder theorem combines the congruences into
a single residue class modulo
\begin{align*}
        [q,\qsiegel]rd.
\end{align*}
Hence, using partial summation to handle the weight
$n^{\widetilde{\beta}-1}$, we get
\begin{align*}
\sum_{\substack{n\leq Y\\ n\equiv b\pmod q\\ r\mid n\\ d\mid n}}\widetilde{\chi}(n)n^{\widetilde{\beta}-1}=
\frac{Y^{\widetilde{\beta}}}{\widetilde{\beta}}
\frac{1}{[q,\qsiegel]rd}
\sum_{\substack{h\pmod {\qsiegel}\\
        h\equiv b\pmod {(q,\qsiegel)}}}
\widetilde{\chi}(h)+O(\qsiegel).
\end{align*}
If $\qsiegel\nmid q$, then $(q,\qsiegel)$ is
a proper divisor of $\qsiegel$, and then
\begin{align*}
        \sum_{\substack{h\pmod {\qsiegel}\\ h\equiv b\pmod {(q,\qsiegel)}}}
        \widetilde{\chi}(h)=0
\end{align*}
by primitivity of $\widetilde{\chi}$. If $\qsiegel\mid q$, then $(q,\qsiegel)=\qsiegel$, so the sum is $\widetilde{\chi}(b)$.

Recalling~\eqref{eq:lambda-twist-expand} and using again the tail estimate~\eqref{eq:smoothtail} for the truncated $r$ sum, we get
\begin{align}\label{eq:model-twisted-principal}
\sum_{\substack{n\leq Y\\ n\equiv b\pmod q\\ (n,m)=1}}
\Lambda_W^{(I)}(n)\widetilde{\chi}(n)n^{\widetilde{\beta}-1}=
1_{\qsiegel\mid q}\,
\frac{\varphi(m)}{m}
\frac{\widetilde{\chi}(b)}{\varphi(q)}
\frac{Y^{\widetilde{\beta}}}{\widetilde{\beta}}
+O\left(Yw^{-1}+\qsiegel\tau(m)w^{s_0}(\log w)\right).
\end{align}
Combining \eqref{eq:model-untwisted-principal} and
\eqref{eq:model-twisted-principal}, we obtain
\begin{align*}
\sum_{\substack{n\leq Y\\ n\equiv b\pmod q\\ (n,m)=1}}
\widetilde{\Lambda}_W^{(I)}(n)=
\frac{\varphi(m)}{m}
\left(
\frac{Y}{\varphi(q)}
-1_{\qsiegel\mid q}\,
\frac{\widetilde{\chi}(b)}{\varphi(q)}
\frac{Y^{\widetilde{\beta}}}{\widetilde{\beta}}
\right)+O\left(Xw^{-1}+\tau(m)w^{s_0+1}(\log w)\right).
\end{align*}
Multiplication by $m/\varphi(m)$ makes the main terms agree with those
obtained from \eqref{eq:Lambda-coprime-m} and \eqref{eq:Lambda-conclusion}. Therefore
\begin{align}\label{eq:principal-sum-bound}
\sum_{\substack{m\leq M'\\ (m,\prod_{p\leq w}p)=1}}
\max_{Y\leq X}
\frac{1}{\varphi(m)}
\left|\sum_{\substack{n\leq Y\\ n\equiv b\pmod q\\ (n,m)=1}}F_m(n)
\right|\ll Xw^{-1}\log N.
\end{align}

It remains to prove
\begin{align}\label{eq:centered-goal}
\sum_{\substack{m\leq M'\\ (m,\prod_{p\leq w}p)=1}}
\max_{\substack{b\pmod {mq}\\ (b,mq)=1}}
\max_{Y\leq X}
\left|D_{q,b,F_m1_{(X^{1/6},X]}}(Y;m,b)\right|
\ll
X\left(\theta^{1/6}+w^{-1/3}\right)(\log N)^{C_0},
\end{align}
where we replaced $F_m$ by $F_m1_{(X^{1/6},X]}$ by estimating by a direct pointwise bound the contribution of $F_m1_{[1,X^{1/6}]}$ by $\ll M'X^{1/6+1/100}(\log X)\ll Xw^{-1/3}$

We decompose $F_m$ into three kinds of terms. We write $h_{\leq Z}(n)=h(n)1_{n\leq Z}$ and $h_{>Z}=h-h_{\leq Z}$. With $U=V=X^{1/6}$, Vaughan's identity gives
\begin{align*}
\Lambda=\Lambda_{\leq V}+\mu_{\leq U}*\log-1*\mu_{\leq U}*\Lambda_{\leq V}+\mu_{>U}*(1*\Lambda_{>V}),
\end{align*}
where each individual convolution factor is $\leq \log X$ on $[1,X]$.
Using this combined with the definition of $\widetilde{\Lambda}^{(I)}$ and the fact that $w^{s_0}\leq X^{1/3}$ for all sufficiently large $N$, we can write $F_m/\log X$ on $(X^{1/6},X]$ as a linear combination of $O((\log X)^{C_0})$ functions of the following forms.

The first kind is
\begin{align*}
        f^{(I)}=1*\alpha \quad\text{or}\quad\log*\alpha
\end{align*}
where $\alpha(n)\ll \tau(n)^{C_0}$ and $\alpha$ is supported on
$[1,2 X^{1/3}]$.

The second kind is the twisted type I form
\begin{align*}
        f^{(I,\mathrm{tw})}(n)= (1*\alpha)(n)\widetilde{\chi}(n)n^{\widetilde{\beta}-1},
\end{align*}
with the same restrictions on $\alpha$.

The third kind is
\begin{align*}
        f^{(II)}=\beta*\gamma,
\end{align*}
where $\beta(n),\gamma(n)\ll \tau(n)^{C_0}$ and $\beta,\gamma$ are supported in
dyadic intervals $(B,2B], (C,2C]$
with
\begin{align*}
        X^{1/6}\ll B,C\ll X^{5/6},\qquad BC\ll X.
\end{align*}
It suffices to prove \eqref{eq:centered-goal} for each of these pieces.

For $f=1*\alpha$, direct counting gives
\begin{align*}
D_{q,b,f}(Y;m,a)
&=\sum_{\ell\leq 2 X^{1/3}}\alpha(\ell)
\left(
\sum_{\substack{n\leq Y\\ n\equiv b\pmod q\\ n\equiv a\pmod m}}
        1_{\ell\mid n}-\frac{1}{\varphi(m)}
\sum_{\substack{n\leq Y\\ n\equiv b\pmod q\\ (n,m)=1}}
        1_{\ell\mid n}
\right)                                                     \\
&\ll \sum_{\ell\leq 2 X^{1/3}}\tau(\ell)^{C_0}\tau(m)
\ll X^{1/3}(\log N)^{C_0}\tau(m),
\end{align*}
where we used that if $(\ell,mq)>1$, both counting functions vanish, and
otherwise the first counting function and the second counting
function divided by $\varphi(m)$ are both
$Y/(q\ell m)+O(\tau(m))$.
Summing over $m\leq M'$ and adjusting $C_0$ yields
\begin{align}\label{eq:Ff}
   \sum_{\substack{m\leq M'\\ (m,\prod_{p\leq w}p)=1}}
\max_{\substack{b\pmod {mq}\\ (b,mq)=1}}
\max_{Y\leq X}
\left|D_{q,b,f}(Y;m,b)\right|  \ll M'X^{1/3}(\log N)^{C_0}
        \ll X\theta^{1/3}(\log N)^{C_0}.
\end{align}

For $f=\log*\alpha$, partial summation gives the same estimate~\eqref{eq:Ff} with an additional factor $\log X$, which is absorbed into $(\log N)^{C_0}$.
The twisted type I terms are treated in the same way. Write
\begin{align*}
 \widetilde{\chi}(n) =\sum_{h\pmod {\qsiegel}}
        \widetilde{\chi}(h)
        1_{n\equiv h\pmod {\qsiegel}}
\end{align*}
and
\begin{align*}
 n^{\widetilde{\beta}-1}
 = 1+(\widetilde{\beta}-1)\int_1^X
        u^{\widetilde{\beta}-2}1_{u\leq n}\,\d u.
\end{align*}
The preceding type I estimate applies to the resulting congruence sums, and partial summation handles the contribution of the integral. This gives a contribution of
\begin{align*}
        \ll \qsiegel M'X^{1/3}(\log N)^{C_0}\ll X\theta^{1/3}(\log N)^{C_0},
\end{align*}
since $\qsiegel\leq w\leq N^{1/3}$ and $X=M'N$.

It remains to estimate a type II piece $f=\beta*\gamma$. Orthogonality of
characters modulo $m$ gives
\begin{align*}
D_{q,b,f}(Y;m,a)=
\frac{1}{\varphi(m)}
\sum_{\substack{\chi\pmod m\\ \chi\neq \chi_0}}
\overline{\chi}(a)
\sum_{\substack{u,v\in \mathbb{Z}\\ uv\leq Y\\ uv\equiv b\pmod q}}\beta(u)\gamma(v)\chi(uv).
\end{align*}
Detecting the congruence $uv\equiv b\pmod q$ by characters modulo $q$, we obtain
\begin{align*}
D_{q,b,f}(Y;m,a)=
\frac{1}{\varphi(m)\varphi(q)}
\sum_{\psi\pmod q}\overline{\psi}(b)
\sum_{\substack{\chi\pmod m\\ \chi\neq \chi_0}}
\overline{\chi}(a)
\sum_{u,v\in \mathbb{Z}} \beta(u)\psi(u)\chi(u)
\gamma(v)\psi(v)\chi(v)1_{uv\leq Y}.
\end{align*}

We remove the condition $uv\leq Y$ by decomposing into short multiplicative
intervals. Let
\begin{align*}
        \eta\coloneqq \frac{1}{100}\left(\theta^{1/6}+w^{-1/3}\right),
\end{align*}
and let
\begin{align*}
        I_i=((1+\eta)^i,(1+\eta)^{i+1}]
\end{align*}
for integers $i\geq 0$. Then
\begin{align*}
1_{uv\leq Y}
=
\sum_{\substack{i,j\geq 0\\ (1+\eta)^{i+j+2}\leq Y}}
        1_{u\in I_i}1_{v\in I_j}
+
O\left(
\sum_{\substack{i,j\geq 0\\
        (1+\eta)^{i+j}\leq Y< (1+\eta)^{i+j+2}}}
        1_{u\in I_i}1_{v\in I_j}
\right).
\end{align*}
Indeed, if $u\in I_i$ and $v\in I_j$, then
\begin{align*}
        (1+\eta)^{i+j}<uv\leq (1+\eta)^{i+j+2}.
\end{align*}
Thus boxes with $(1+\eta)^{i+j+2}\leq Y$ lie inside the region
$uv\leq Y$, while the only boxes not decided by this condition are those with
\begin{align*}
        (1+\eta)^{i+j}\leq Y<(1+\eta)^{i+j+2}.
\end{align*}

We next estimate the contribution of the boundary boxes. If
\begin{align*}
        (1+\eta)^{i+j}\leq Y< (1+\eta)^{i+j+2},
\end{align*}
then every pair $(u,v)\in I_i\times I_j$ satisfies
\[
(1+\eta)^{-2}Y<uv\leq(1+\eta)^2Y.
\]
We have
\[
\sum_{uv=t}|\beta(u)\gamma(v)|\ll\tau(t)^{C_0}.
\]
For $Y\leq X^{3/4}$, the divisor bounds and direct counting give a total contribution to \eqref{eq:centered-goal} of
\begin{align*}
\ll X^{1/100}\sum_{m\leq M'}\left(\frac{X^{3/4}}{mq}+1+\frac{X^{3/4}}{\varphi(m)}\right)\ll X^{4/5}\ll  Xw^{-1/3}.
\end{align*}
We may therefore assume that $Y>X^{3/4}$. Then $mq\leq X^{1/2}w\leq Y^{3/4}$ and $\eta Y\geq Y^{9/10}$ for sufficiently large $N$. Now, Shiu's bound in short arithmetic progressions, followed by summation over $m\leq M'$, therefore shows that the contribution of the boundary boxes to~\eqref{eq:centered-goal} is
\[
\ll\eta X(\log X)^{C_0} \ll X\left(\theta^{1/6}+w^{-1/3}\right)(\log N)^{C_0},
\]
which is admissible.

It remains to estimate the contribution of boxes lying entirely inside
$uv\leq Y$. For such a box, write
\begin{align*}
        \beta_1(u)=\beta(u)1_{u\in I_i},\qquad
        \gamma_1(v)=\gamma(v)1_{v\in I_j}.
\end{align*}
The condition $uv\leq Y$ is then absent. The number of boxes is
\begin{align}\label{eq:boxnumber}
        O\left(\eta^{-2}(\log X)^2\right)
        \ll\min\{\theta^{-1/3},w^{2/3}\}(\log X)^2.
\end{align}

Thus it is enough to estimate, for each such box, sums of the form
\begin{align*}
\frac{1}{\varphi(m)\varphi(q)}
\sum_{\psi\pmod q}
\sum_{\substack{\chi\pmod m\\ \chi\neq \chi_0}}
\left|\sum_{u\in \mathbb{Z}}\beta_1(u)\psi(u)\chi(u)\right|
\left|\sum_{v\in \mathbb{Z}}\gamma_1(v)\psi(v)\chi(v)\right|.
\end{align*}
Here $\beta_1$ and $\gamma_1$ are restrictions of $\beta$ and $\gamma$ to short multiplicative intervals contained in the original dyadic supports
$B<u\leq 2B$ and $C<v\leq 2C$.

Let $\chi$ be induced by the primitive character $\chi^\ast\pmod r$.
Since $(m,\prod_{p\leq w}p)=1$ and $\chi$ is non-principal, every prime
divisor of $r$ is larger than $w$, so $r>w$. Writing $m=rs$, we have
\begin{align*}
        \chi(n)=\chi^\ast(n)1_{(n,s)=1}.
\end{align*}
For each $s\geq1$, define
\begin{align*}
\beta_{2,s}(u)=\beta_1(u)1_{(u,s)=1},\qquad  \gamma_{2,s}(v)=\gamma_1(v)1_{(v,s)=1}.
\end{align*}
We use $\varphi(rs)^{-1}\ll(\log X)/(rs)$ and split $r>w$ into dyadic intervals $R<r\leq2R$, where $R=2^jw$ with $j\geq0$ and $R\leq M'$. Enlarging the ranges of the nonnegative sums, the type II contribution is
\begin{align}\label{eq:typeII-reduction}
\ll
\frac{(\log X)^{C_0}}{\varphi(q)}
\sum_{\psi\pmod q}
\sum_R\sum_{s\leq M'/R}\frac1s\sum_{R<r\leq2R}
\frac{1}{r}
\sideset{}{^{\ast}}\sum_{\chi\pmod r}
\left|\sum_u \beta_{2,s}(u)\psi(u)\chi(u)\right|
\left|\sum_v \gamma_{2,s}(v)\psi(v)\chi(v)\right|,
\end{align}
where
\begin{align*}
        \|\beta_{2,s}\|_2\ll \|\beta_1\|_2(\log X)^{C_0},\qquad
        \|\gamma_{2,s}\|_2\ll \|\gamma_1\|_2(\log X)^{C_0}.
\end{align*}

For each $R$ and $s$, applying the Cauchy--Schwarz inequality and the large sieve inequality from~\cite[(7.31)]{iw-kow} gives
\begin{align*}
&\sum_{R<r\leq 2R}
\frac{1}{r}
\sideset{}{^{\ast}}\sum_{\chi\pmod r}
\left|\sum_u \beta_{2,s}(u)\psi(u)\chi(u)\right|
\left|\sum_v \gamma_{2,s}(v)\psi(v)\chi(v)\right|                 \\
&\qquad\ll
\|\beta_{2,s}\|_2\|\gamma_{2,s}\|_2
R^{-1}(R^2+B)^{1/2}(R^2+C)^{1/2}\log X.
\end{align*}
By a standard estimate for moments of the divisor function, we have
\begin{align*}
        \|\beta_{2,s}\|_2\|\gamma_{2,s}\|_2
        \ll (BC)^{1/2}(\log X)^{C_0}
        \ll X^{1/2}(\log X)^{C_0},
\end{align*}
so, using $\sum_{s\leq M'/R}s^{-1} \ll \log X$ and increasing $C_0$, the contribution of the dyadic block $(R,2R]$ is
\begin{align*}
 \ll X^{1/2}(\log X)^{C_0}
 R^{-1}(R^2+B)^{1/2}(R^2+C)^{1/2}.
\end{align*}
Using
\begin{align*}
 R^{-1}(R^2+B)^{1/2}(R^2+C)^{1/2}
 \ll
 R+B^{1/2}+C^{1/2}+\frac{(BC)^{1/2}}{R},
\end{align*}
and summing over dyadic $R$ with $w\leq R\leq M'$, we get a contribution of
\begin{align*}
&\ll
X^{1/2}(\log X)^{C_0}
\left( M'+X^{5/12}+\frac{X^{1/2}}{w}
\right)                                                        \\
&\ll
X(\log X)^{C_0}
\left(\theta^{1/2}+X^{-1/12}+w^{-1}
\right).
\end{align*}
After multiplying by the number of boxes in~\eqref{eq:boxnumber}, this is
\begin{align*}
&\ll
X(\log X)^{C_0}
\min\{\theta^{-1/3},w^{2/3}\}
\left(\theta^{1/2}+X^{-1/12}+w^{-1}\right) \\
&\ll
X\left(\theta^{1/6}+w^{-1/3}\right)(\log N)^{C_0}.
\end{align*}
since $X\geq N$ and $w\leq N^{1/12}$ for $N$ sufficiently large. This proves \eqref{eq:centered-goal}. Combining this with
\eqref{eq:principal-sum-bound} and the replacement of
$\widetilde{\Lambda}_W$ by $\widetilde{\Lambda}_W^{(I)}$ proves the lemma.
\end{proof}

\section{Proof of first main estimate}
\label{sec_proof_of_lem_error_estimate_gowers_norm_bound}

The purpose of this section is to prove 
Proposition~\ref{prop_error_estimate_gowers_norm_bound}.
The proof has two main ingredients. The first is the local Gowers norm bound for $\Lambda-\widetilde{\Lambda}_W$ given by Theorem~\ref{lem_local_uniformity_norm_bound_for_siegel-corrected_vonMangoldt}. 
The other is the following standard estimate, which controls multicorrelations along linear forms in terms of Gowers norms. The version that we use is implicit in \cite{gt-linear} and is a special case of \cite[Proposition~1]{CGSS21}.

\begin{proposition}[Generalised von Neumann theorem for cyclic groups]
\label{prop_GvN}
Let $p$ be a prime, let $(u_1,v_1),\ldots,(u_r,v_r)\in (\mathbb{Z}/p\mathbb{Z})^2\setminus \{(0,0)\}$, and suppose that $u_iv_j-u_jv_i\not\equiv 0\pmod p$ whenever $i\neq j$. Then, for every collection of $1$-bounded functions $f_1,\ldots,f_r\colon \Z/p\Z\to\mathbb{C}$, we have
\[
\left|
\sum_{m,n\in(\Z/p\Z)}
\prod_{j=1}^r
f_j\bigl(u_jm+v_jn\bigr)
\right|
\leq p^2\cdot 
\min_{1\leq j\leq r}
\|f_j\|_{U^{r+1}(\Z/p\Z)}.
\]
\end{proposition}

\begin{proof}[Proof of Proposition~\ref{prop_error_estimate_gowers_norm_bound}]
Suppose $\mathbf{0}\neq\mathbf{t}=(t_1,\ldots,t_k)\in \{0,1\}^k$ is given.
Recall that
\[
S_{\mathrm{G},\mathbf{t}}= \sum_{d\mid \mathcal{D}}\sum_{\substack{m\leq M\\ (m,W_1)=1}}\sum_{n\leq N}\Lambda(\psi_{-1}(dm))\Lambda(\psi_0(\mathcal{D}mn))
\prod_{j=1}^k F_{t_j}\big(\psi_j\big(dm,\tfrac{\mathcal{D}}{d}n\big)\big),
\]
where
$F_{0}(n)=\widetilde{\Lambda}_{W_2}(n)$ and $F_{1}(n)=\big(\Lambda-\widetilde{\Lambda}_{W_2}\big)(n)$, 
and our goal is to show
\begin{equation}
\label{eqn_error_estimate_gowers_norm_bound_1}
|S_{\mathrm{G},\mathbf{t}}|\ll  \frac{\tau(\mathcal{D})MN}{(\log N)^{10k}}.
\end{equation}
By the Cauchy--Schwarz inequality and the crude bound $\Lambda(\psi_0(\mathcal{D}mn))\ll \log N$ for large enough $N$, we have
\begin{align*}
|S_{\mathrm{G},\mathbf{t}}|^2 
&\ll \left(\sum_{\substack{m\leq M\\ (m,W_1)=1}}\sum_{n\leq N}\left| \sum_{d\mid \mathcal{D}} \Lambda(\psi_{-1}(dm))
\prod_{j=1}^k F_{t_j}\big(\psi_j\big(dm,\tfrac{\mathcal{D}}{d}n\big)\big)\right|^2\right)
\\
&\hspace{6cm}\cdot\left(\sum_{\substack{m\leq M\\ (m,W_1)=1}}\sum_{n\leq N}\left| \Lambda(\psi_0(\mathcal{D}mn))\right|^2\right)
\\
&\ll \left(MN (\log N)^2\right)\cdot \sum_{\substack{m\leq M\\ (m,W_1)=1}}\sum_{n\leq N}\left| \sum_{d\mid \mathcal{D}} \Lambda(\psi_{-1}(dm))
\prod_{j=1}^k F_{t_j}\big(\psi_j\big(dm,\tfrac{\mathcal{D}}{d}n\big)\big)\right|^2
\\
&\ll \left(MN (\log N)^2\right)\cdot \sum_{m\leq M}\sum_{n\leq N}\left| \sum_{d\mid \mathcal{D}} \Lambda(\psi_{-1}(dm))
\prod_{j=1}^k F_{t_j}\big(\psi_j\big(dm,\tfrac{\mathcal{D}}{d}n\big)\big)\right|^2
\\
&\ll \left(MN (\log N)^2\right)\cdot \sum_{m\leq M}\sum_{N^{9/10}\leq n\leq N}\left| \sum_{d\mid \mathcal{D}} \Lambda(\psi_{-1}(dm))
\prod_{j=1}^k F_{t_j}\big(\psi_j\big(dm,\tfrac{\mathcal{D}}{d}n\big)\big)\right|^2
\\
&\hspace{.62\textwidth} + M^2N^{19/10}\tau(\mathcal D)^2 (\log N)^{2k+4}.
\end{align*}
Write $\Sigma_1$ for the double sum
\[
\Sigma_1\coloneqq\sum_{m\leq M}\sum_{N^{9/10}\leq n\leq N}\left| \sum_{d\mid \mathcal{D}} \Lambda(\psi_{-1}(dm))
\prod_{j=1}^k F_{t_j}\big(\psi_j\big(dm,\tfrac{\mathcal{D}}{d}n\big)\big)\right|^2.
\]
Thus, instead of \eqref{eqn_error_estimate_gowers_norm_bound_1} it suffices to show
\begin{equation}
\label{eqn_error_estimate_gowers_norm_bound_2}
\Sigma_1\ll  \frac{\tau(\mathcal{D})^2 MN}{(\log N)^{20k+2}}.
\end{equation}

By partitioning $[N^{9/10},N]$ into intervals of length $M$, and enlarging the last interval to length $M$, we have
\[
\Sigma_1\ll \frac{N}{M}\sup_{\substack{x\in\mathbb Z\\N^{9/10}-1\leq x\leq N}} \left(\sum_{m,n\leq M}\left| \sum_{d\mid \mathcal{D}} \Lambda(\psi_{-1}(dm))
\prod_{j=1}^k F_{t_j}\big(\psi_j\big(dm,\tfrac{\mathcal{D}}{d}(x+n)\big)\big)\right|^2\right).
\]
Write $\Sigma_2(x)$ for the sum inside the parentheses. 
Then~\eqref{eqn_error_estimate_gowers_norm_bound_2} reduces to showing that for $N^{9/10}-1\leq x\leq N$,
\begin{equation}
\label{eqn_error_estimate_gowers_norm_bound_3}
\Sigma_2(x)\ll  \frac{\tau(\mathcal{D})^2 M^2}{(\log N)^{20k+2}}.
\end{equation}

Fix $x\in [N^{9/10}-1,N]$. Let $\mathcal{B}$ be the set of $(d_1,d_2)$ for which $d_1,d_2\mid \mathcal{D}$, and at least two of the $2k$ linear forms $L_j(\Delta d_1m,\frac{\Delta \mathcal{D}}{d_1}n)$, $L_j(\Delta d_2m,\frac{\Delta \mathcal{D}}{d_2}n)$ are proportional. Forms corresponding to the same divisor are non-proportional by the hypotheses on $L_1,\ldots,L_k$. If forms corresponding to $j,d_1$ and $j',d_2$ are proportional, then
\[
a_jb_{j'}d_1^2=a_{j'}b_jd_2^2.
\]
For fixed $j,j'$ and $d_2$, this determines at most one $d_1$. Consequently,
\[
|\mathcal B|\leq k^2\tau(\mathcal D).
\] Expanding the square in $\Sigma_2(x)$ gives
\begin{align*}
\Sigma_2(x)&= \sum_{d_1,d_2\mid \mathcal{D}}  \Bigg(\sum_{m,n\leq M} \Lambda(\psi_{-1}(d_1m))\Lambda(\psi_{-1}(d_2m))
\\
&\hspace{.3\textwidth}\cdot
\prod_{j=1}^k F_{t_j}\big(\psi_j\big(d_1m,\tfrac{\mathcal{D}}{d_1}(x+n)\big)\big) F_{t_j}\big(\psi_j\big(d_2m,\tfrac{\mathcal{D}}{d_2}(x+n)\big)\big) \Bigg)
\\
&=\sum_{\substack{d_1,d_2\mid \mathcal{D}\\ (d_1,d_2)\not \in \mathcal{B}}}  \Bigg(\sum_{m,n\leq M} \Lambda(\psi_{-1}(d_1m))\Lambda(\psi_{-1}(d_2m))
\\
&\hspace{.03\textwidth}\cdot
\prod_{j=1}^k F_{t_j}\big(\psi_j\big(d_1m,\tfrac{\mathcal{D}}{d_1}(x+n)\big)\big) F_{t_j}\big(\psi_j\big(d_2m,\tfrac{\mathcal{D}}{d_2}(x+n)\big)\big) \Bigg)+ \mathrm{O}\left(\tau(\mathcal{D})M^2 (\log N)^{2k+2}\right).
\end{align*}

We have
\[
\tau(\mathcal D)=(D+1)^\ell
\geq(\log N)^{\delta \ell}.
\]
Our choice of $\ell$ (see \eqref{def_ck_and_ell}) gives $\delta \ell>22k+4$. Therefore
\[
\tau(\mathcal D)M^2(\log N)^{2k+2}
\ll
\frac{\tau(\mathcal D)^2M^2}{(\log N)^{20k+2}}.
\]
Hence, instead of \eqref{eqn_error_estimate_gowers_norm_bound_3}, it is enough to show that when $d_1\mid \mathcal{D}$, $d_2\mid \mathcal{D}$ and $(d_1,d_2)\not \in \mathcal{B}$, we have
\begin{equation}
\label{eqn_error_estimate_gowers_norm_bound_4}
\begin{aligned}
\Sigma_3(x)\coloneqq \sum_{m,n\leq M} &\Lambda(\psi_{-1}(d_1m)) \Lambda(\psi_{-1}(d_2m))
\\
\cdot & \prod_{j=1}^k F_{t_j}\big(\psi_j\big(d_1m,\tfrac{\mathcal{D}}{d_1}(x+n)\big)\big) F_{t_j}\big(\psi_j\big(d_2m,\tfrac{\mathcal{D}}{d_2}(x+n)\big)\big)
\ll  \frac{M^2}{(\log N)^{20k+2}}.&
\end{aligned}
\end{equation}

Fix $d_1,d_2\mid \mathcal{D}$ with $(d_1,d_2)\not \in \mathcal{B}$.
By definition, we have
\[
\psi_j\big(d_im, \tfrac{\mathcal{D}}{d_i}(x+n)\big)=L_j\big(\Delta d_im,\tfrac{\Delta\mathcal D}{d_i}n\big)+\tfrac{\Delta\mathcal{D}}{d_i}b_jx+1.
\]
Note also that
\[
L_j\big(\Delta d_im,\tfrac{\Delta\mathcal D}{d_i}n\big)\in [4\Delta^2\mathcal{D}M] \qquad\textnormal{ for all }n,m\leq M.
\]
Let $p$ be any prime number with $2^{2k+6}\Delta^2\mathcal{D}M< p \leq 2^{2k+7}\Delta^2\mathcal{D}M$. 
Letting
\begin{align*}
G(m)&=1_{[M]}(m) \Lambda(\psi_{-1}(d_1m))\Lambda(\psi_{-1}(d_2m)), 
\\
G_{j}^{(i)}(y)&= F_{t_j}(y+\tfrac{\Delta\mathcal{D}}{d_i}b_jx+1) 1_{[4\Delta^2\mathcal{D}M]}(y),
\end{align*}
we can embed the sum in $\Sigma_3(x)$ into $\Z/p\Z$ and have
\[
\Sigma_3(x)=\sum_{n,m\in\mathbb Z/p\mathbb Z}
1_{[M]}(n)G(m)\prod_{j=1}^k G_j^{(1)}\left(L_j\left(\Delta d_1m,\frac{\Delta\mathcal D}{d_1}n\right)
\right)G_j^{(2)}\left(L_j\left(\Delta d_2m,\frac{\Delta\mathcal D}{d_2}n\right)
\right).
\]
Since $(d_1,d_2)\not \in \mathcal{B}$, the linear forms
\[
n,\quad m,\quad
L_j\big(\Delta d_1m,\tfrac{\Delta\mathcal D}{d_1}n\big),\quad
L_j\big(\Delta d_2m,\tfrac{\Delta\mathcal D}{d_2}n\big),
\qquad 1\leq j\leq k,
\]
are pairwise non-proportional over $\mathbb Q$. Their coefficient determinants are non-zero integers of size $O(\mathcal D^2)$. Since $\mathcal D^2=o(p)$, they remain non-zero modulo $p$. 

Choose $j_0$ such that $t_{j_0}=1$, and let $K=C\log N$ with $C$ sufficiently large. We have
\[
\|G\|_\infty\leq K^2
\qquad\text{and}\qquad
\|G_j^{(i)}\|_\infty\leq K.
\]
Applying Proposition~\ref{prop_GvN} to the normalized functions gives
\[
\Sigma_3(x)\ll
p^2K^{2k+1}
\left\|G_{j_0}^{(1)}\right\|_{U^{2k+3}(\mathbb Z/p\mathbb Z)}.
\]
Let
\[
I=\frac{\Delta\mathcal D}{d_1}b_{j_0}x+1+[4\Delta^2\mathcal DM].
\]
By the definition of the interval Gowers norm,
\begin{align*}
\left\|G_{j_0}^{(1)}\right\|_{U^{2k+3}(\mathbb Z/p\mathbb Z)}
&=\left\|1_{[4\Delta^2\mathcal DM]}\right\|_
 {U^{2k+3}(\mathbb Z/p\mathbb Z)}\left\|F_{t_{j_0}}\right\|_{U^{2k+3}(I)}\\
&\leq\left\|\Lambda-\widetilde\Lambda_{W_2}\right\|_{U^{2k+3}(I)}.
\end{align*}
Consequently,
\[
\Sigma_3(x)\ll p^2(\log N)^{2k+1}\left\|\Lambda-\widetilde\Lambda_{W_2}\right\|_{U^{2k+3}(I)}.
\]

We apply Theorem~\ref{lem_local_uniformity_norm_bound_for_siegel-corrected_vonMangoldt} with
\[
X=\frac{\Delta\mathcal D}{d_1}b_{j_0}x+1\qquad\text{and}\qquad
H=4\Delta^2\mathcal DM.
\]
For sufficiently large $N$, we have $N^{9/10}\leq X\leq N^{1+o(1)}$, $X^{11/16}\leq H\leq X$, and $w_2\leq\exp((\log X)^{1/10})$. The theorem gives, for some $\gamma=\gamma(k)>0$,
\[
\left\|\Lambda-\widetilde{\Lambda}_{W_2}\right\|_{U^{2k+3}(I)}\ll w_2^{-a_{2k+3}}+\exp\bigl(-(\log X)^{\gamma}\bigr)
\ll\exp\bigl(-(\log N)^{2\delta}\bigr).
\]
Using $p\ll\mathcal DM$, the conclusion of \eqref{eqn_error_estimate_gowers_norm_bound_4} follows.
\end{proof}

\section{Proof of second main estimate}
\label{sec_proof_of_lem_error_estimate_BV_bound}

This section is dedicated to the proof of Proposition~\ref{prop_error_estimate_bombieri-vinogradov_bound}. The main technical ingredient in this proof is Lemma~\ref{le:BV}.

\begin{proof}[Proof of Proposition~\ref{prop_error_estimate_bombieri-vinogradov_bound}]
Our goal is to show that
\begin{align}
\label{eqn_error_estimate_bombieri-vinogradov_bound_pr_1}
|S_{\mathrm{BV}}|
&\ll \tau(\mathcal{D})MN\exp(-(\log N)^{1/20}).
\end{align}
where
\begin{align*}
S_{\mathrm{BV}}&=
\sum_{d\mid \mathcal{D}}\sum_{\substack{m\leq M\\ (m,W_1)=1}}\sum_{n\leq N} \Lambda(\psi_{-1}(dm)) \big(\Lambda-\tfrac{m}{\varphi(m)}\widetilde{\Lambda}_{W_1}\big)(\psi_0(\mathcal{D}mn))\prod_{j=1}^k \widetilde{\Lambda}_{W_2}\big(\psi_j\big(dm,\tfrac{\mathcal{D}}{d}n\big)\big)
\end{align*}

Let $\Lambda_{W_2}^{(I)}(n)$ be as in \eqref{def_typeI-sum} with $s=1000(\log N)^{1/10}$ and define
\[
\widetilde{\Lambda}_{W_2}^{(I)}(n)\coloneqq \Lambda_{W_2}^{(I)}(n)\bigl(1-\widetilde{\chi}_2(n)n^{\widetilde{\beta}_2-1}\bigr).
\]
For $d\mid\mathcal D$, $m\leq M$, and $1\leq j\leq k$, write
\[
A_j=\frac{\Delta\mathcal D b_j}{d}\qquad\text{and}\qquad B_j=\Delta a_jdm+1.
\]
We apply Lemma~\ref{le:Lambdareplace} with
\[
w=w_2,\qquad q=A_j,\qquad a=B_j,\qquad X=A_jN+B_j.
\]
Its hypotheses hold for sufficiently large $N$, since every prime divisor of $\Delta b_j$ divides $\Delta$ and hence does not divide $B_j$, so $(A_j,B_j)\mid\mathcal D=\exp(o(s))$, while $X\geq N\geq w_2^{10s}$ and $X\geq A_j^{10/9}$. It follows that
\begin{align}\begin{split}\label{eqn_single_factor_l1}
\sum_{n\leq N}\big| \widetilde{\Lambda}_{W_2}\big(\psi_j\big(dm,\tfrac{\mathcal{D}}{d}n\big)\big) - & \widetilde{\Lambda}_{W_2}^{(I)}\big(\psi_j\big(dm,\tfrac{\mathcal{D}}{d}n\big)\big)\big|
\\
&=\sum_{n\leq N}\big| \widetilde{\Lambda}_{W_2}\big(\Delta a_jdm+\tfrac{\Delta \mathcal{D}b_j}{d}n+1\big) - \widetilde{\Lambda}_{W_2}^{(I)}\big(\Delta a_jdm+\tfrac{\Delta \mathcal{D}b_j}{d}n+1\big)\big|
\\
&\ll\sum_{n\leq N}\big| \Lambda_{W_2}\big(\Delta a_jdm+\tfrac{\Delta \mathcal{D}b_j}{d}n+1\big) - \Lambda_{W_2}^{(I)}\big(\Delta a_jdm+\tfrac{\Delta \mathcal{D}b_j}{d}n+1\big)\big|
\\
&\ll Ne^{-s/12}.
\end{split}
\end{align}
For $1\leq j\leq k$, write
\begin{align*}
\Psi_j(n)&=\psi_j\big(dm,\tfrac{\mathcal D}{d}n\big)=A_jn+B_j,\\
F_j(n)&=\widetilde\Lambda_{W_2}(\Psi_j(n)),\\
F_j^{(I)}(n)&=\widetilde\Lambda_{W_2}^{(I)}(\Psi_j(n)),\\
E_j(n)&=F_j(n)-F_j^{(I)}(n).
\end{align*}
Thus \eqref{eqn_single_factor_l1} gives
\begin{align}
\label{eqn_single_factor_error}
\sum_{n\leq N}|E_j(n)|\ll Ne^{-s/12}.
\end{align}

Note that
\begin{align*}
|F_j(n)|&\ll\log w_2,\\
|F_j^{(I)}(n)|
&\ll(\log w_2)
\sum_{\substack{r\mid W_2\\r\leq w_2^s}}1_{r\mid\Psi_j(n)}.
\end{align*}
After expanding the $h$-th power in the second estimate for a fixed positive integer $h$, the arising congruence $[r_1,\ldots, r_h]\mid A_jn+B_j$ either has no solution or defines one residue class modulo $[r_1,\ldots, r_h]/(A_j,[r_1,\ldots, r_h])$. In the latter case $(A_j,[r_1,\ldots, r_h])\mid(A_j,B_j)$. Since $[r_1,\ldots, r_h]\mid W_2$ is squarefree, we have 
\[
(A_j,[r_1,\ldots, r_h])\leq \textnormal{rad}((A_j,B_j))
\leq\textnormal{rad}(\mathcal D)
\ll(\log N)^{O_k(1)}.
\]
Therefore we have
\begin{align*}
\sum_{n\leq N}|F_j^{(I)}(n)|^h
&\ll(\log w_2)^h \sum_{\substack{r_1,\ldots,r_h\mid W_2\\r_1,\ldots, r_h\leq w_2^s}} \left(N\frac{(A_j,[r_1,\ldots, r_h])}{[r_1,\ldots, r_h]}+1\right)\\
&\ll(\log w_2)^h\left(N(\log N)^{O_k(1)}\prod_{p\leq w_2}\left(1+\frac{2^h-1}{p}\right)+w_2^{sh}\right)\\
&\ll_{h}N(\log N)^{C_{h,k}},
\end{align*}
for some $C_{h,k}>0$, where we used $w_2^{sh}=N^{o(1)}$. This implies that
\begin{align}
\label{eqn_factor_moments}
\sum_{n\leq N}\left(|F_j(n)|^h+|F_j^{(I)}(n)|^h\right)\ll_{h}N(\log N)^{C_{h,k}}.
\end{align}

Taking $h=2k-1$ in \eqref{eqn_factor_moments} and applying the
Cauchy--Schwarz inequality, we obtain
\begin{align}
\label{eqn_factor_Lk}
\sum_{n\leq N}|E_j(n)|^k
&\leq \left(\sum_{n\leq N}|E_j(n)|\right)^{1/2} \left(\sum_{n\leq N}|E_j(n)|^{2k-1}\right)^{1/2}\notag\\
&\ll Ne^{-s/24}(\log N)^{C_k}
\end{align}
for some $C_k>0$. For every $n\in \mathbb{N}$ we have the identity
\[
\prod_{j=1}^kF_j(n)-\prod_{j=1}^kF_j^{(I)}(n)=\sum_{i=1}^k E_i(n)\prod_{j<i}F_j(n)\prod_{j>i}F_j^{(I)}(n).
\]
For each $1\leq i\leq k$, H\"older's inequality, \eqref{eqn_factor_moments} with $h=k$, and \eqref{eqn_factor_Lk} give
\begin{align*}
&\sum_{n\leq N}|E_i(n)| \prod_{j<i}|F_j(n)|\prod_{j>i}|F_j^{(I)}(n)|\\
&\qquad\leq \left(\sum_{n\leq N}|E_i(n)|^k\right)^{1/k} \prod_{j<i}\left(\sum_{n\leq N}|F_j(n)|^k\right)^{1/k} \prod_{j>i}\left(\sum_{n\leq N}|F_j^{(I)}(n)|^k\right)^{1/k}\\
&\qquad\ll_k Ne^{-s/(24k)}(\log N)^{C_k} \ll_k Ne^{-s/(25k)}.
\end{align*}
Summing over $i$ gives
\begin{align}
\label{eqn_approx_lf_1}
\sum_{n\leq N}\left|\prod_{j=1}^k\widetilde\Lambda_{W_2}(\Psi_j(n))-\prod_{j=1}^k\widetilde\Lambda_{W_2}^{(I)}(\Psi_j(n))
\right|
\ll_k Ne^{-s/(25k)}.
\end{align}

Finally, $(m,W_1)=1$ implies $m/\varphi(m)\ll1$, while
\[
\Lambda(\psi_{-1}(dm))\ll\log N \quad\text{and}\quad
\left|\left(\Lambda-\frac{m}{\varphi(m)}\widetilde\Lambda_{W_1}\right)(\psi_0(\mathcal Dmn))\right|\ll\log N.
\]
Therefore the total error caused by this replacement in
$S_{\mathrm{BV}}$ is
\[
\ll\tau(\mathcal D)MN(\log N)^2e^{-s/(25k)}\ll\tau(\mathcal D)MN\exp(-(\log N)^{1/20}).
\]

Therefore, instead of \eqref{eqn_error_estimate_bombieri-vinogradov_bound_pr_1}, it suffices to show that for any fixed $d\mid \mathcal{D}$ we have
\begin{align}
\label{eqn_error_estimate_bombieri-vinogradov_bound_pr_2}
\begin{aligned}
\Bigl|\sum_{\substack{m\leq M\\ (m,W_1)=1}}\sum_{n\leq N} \Lambda(\psi_{-1}(dm)) \big(\Lambda-\tfrac{m}{\varphi(m)}\widetilde{\Lambda}_{W_1}\big)(\psi_0(\mathcal{D}mn)) \prod_{j=1}^k & \widetilde{\Lambda}_{W_2}^{(I)}\big(\psi_j\big(dm,\tfrac{\mathcal{D}}{d}n\big)\big)\Bigr|
\\
&\ll MN\exp(-(\log N)^{1/20}).
\end{aligned}
\end{align}
As
\[
\widetilde{\Lambda}_{W_2}^{(I)}(n)= (1-\widetilde{\chi}_2(n)n^{\widetilde{\beta}_2-1}) \Lambda_{W_2}^{(I)}(n),
\]
we can write the product
\[
\prod_{j=1}^k \widetilde{\Lambda}_{W_2}^{(I)}\big(\psi_j\big(dm,\tfrac{\mathcal{D}}{d}n\big)\big)
=
\sum_{(t_1,\ldots,t_k)\in \{0,1\}^k} (-1)^{t_1+\ldots+t_k}
z_{t_1,\ldots,t_k}(n,m)\, 
v_{t_1,\ldots,t_k}(n,m),
\]
where
\begin{align*}
z_{t_1,\ldots,t_k}(n,m)&=\prod_{j=1}^k \big(\psi_j\big(dm,\tfrac{\mathcal{D}}{d}n\big)\big)^{t_j(\widetilde{\beta}_2-1)},
\\
v_{t_1,\ldots,t_k}(n,m)&=\Bigg(\prod_{j=1}^k \widetilde{\chi}_2\big(\psi_j\big(dm,\tfrac{\mathcal{D}}{d}n\big)\big)^{t_j} \Lambda_{W_2}^{(I)}\big(\psi_j\big(dm,\tfrac{\mathcal{D}}{d}n\big)\big) \Bigg).
\end{align*}
Therefore, instead of \eqref{eqn_error_estimate_bombieri-vinogradov_bound_pr_2}, it is enough to prove that for any $(t_1,\ldots,t_k)\in \{0,1\}^k$ we have
\begin{align}
\label{eqn_error_estimate_bombieri-vinogradov_bound_pr_3}
\begin{aligned}
\Bigl|\sum_{\substack{m\leq M\\ (m,W_1)=1}}\sum_{n\leq N} \Lambda(\psi_{-1}(dm)) \big(\Lambda-\tfrac{m}{\varphi(m)}\widetilde{\Lambda}_{W_1}\big)(\psi_0(\mathcal{D}mn))  z_{t_1,\ldots,t_k} &(n,m)\,  v_{t_1,\ldots,t_k}(n,m)\Bigr|
\\
&\ll MN\exp(-(\log N)^{1/20}).
\end{aligned}
\end{align}
Note that $(n,m)\mapsto z_{t_1,\ldots,t_k}(n,m)$ is non-negative, non-increasing, and convex, with $\|z_{t_1,\ldots,t_k}\|_\infty\leq 1$.
Therefore, using summation by parts in two variables, we can further reduce \eqref{eqn_error_estimate_bombieri-vinogradov_bound_pr_3} to showing that for all $N'\leq N$ and $M'\leq M$ we have
\begin{align}
\label{eqn_error_estimate_bombieri-vinogradov_bound_pr_4}
\begin{aligned}
\Bigl|\sum_{\substack{m\leq M'\\ (m,W_1)=1}}\sum_{n\leq N'} \Lambda(\psi_{-1}(dm)) \big(\Lambda-\tfrac{m}{\varphi(m)}\widetilde{\Lambda}_{W_1}\big)(\psi_0(\mathcal{D}mn))    v_{t_1,\ldots,t_k}(n,m)\Bigr|
\ll MN\exp(-(\log N)^{1/20}).
\end{aligned}
\end{align}

Using the definition \eqref{def_typeI-sum} and using 
$$\widetilde\chi_2(n)=\sum_{b\pmod{\qsiegel_2}}\widetilde\chi_2(b)1_{n\equiv b\pmod{\qsiegel_2}},$$
we can write for every fixed $m$ the function 
$n\mapsto \frac{\varphi(W_2)^k}{W_2^k}v_{t_1,\ldots,t_k}(n,m)$ as a linear combination of indicator functions of arithmetic progressions, where the number of progressions and their moduli are bounded by $\ll \qsiegel_2^k w_2^{ks}\leq w_2^{2ks}$, the coefficients of the linear combination are from $\{-1,0,+1\}$, and
all moduli are $w_2$-smooth. For each term of the expansion, the congruences give either no solutions or a single residue class $n\equiv r_m\pmod q$, where $q$ is independent of $m$.

Since $\frac{W_2^k}{\varphi(W_2)^k}\ll \log(w_2)^k$, to finish the proof it therefore suffices to establish that for all $w_2$-smooth integers $q$ with $1\leq q\leq w_2^{2ks}$, and all $r_m\in \{0,1,\ldots,q-1\}$, $1\leq m\leq M$, we have
\begin{align}
\label{eqn_error_estimate_bombieri-vinogradov_bound_pr_5}
\begin{aligned}
\Bigl|\sum_{\substack{m\leq M'\\ (m,W_1)=1}}\sum_{n\leq N'} \Lambda(\psi_{-1}(dm)) \big(\Lambda-\tfrac{m}{\varphi(m)}\widetilde{\Lambda}_{W_1}\big)(\psi_0(\mathcal{D}mn))  & 1_{n\equiv r_m\pmod q}\Bigr|
\\
&\ll w_2^{-3ks}MN\exp(-(\log N)^{1/20}).
\end{aligned}
\end{align}
Since
\[
(\log N)w_1^{-1/4} \leq w_2^{-3ks}\exp(-(\log N)^{1/20}),
\]
the estimate follows from the next lemma with $g(u)=\Lambda(u)1_{u\leq\Delta dM+1}$, since $\|g\|_\infty\ll\log N$.
\end{proof}

\begin{lemma} Let $N\geq 2$ and $s=1000(\log N)^{1/10}$. Let $1\leq q\leq w_2^{2sk}$ be a $w_2$-smooth integer, and let $r_m\in \{0,1,\ldots,q-1\}$ for $1\leq m \leq M$.
For any function $g\colon \mathbb{N}\to \mathbb{C}$, any $M'\leq M, N'\leq N$, and any $d\mid\mathcal{D}$ we have
\begin{align}
\label{eq:Lambdaclaim1}
\begin{aligned}
\sum_{\substack{m\leq M'\\ (m,W_1)=1}}\sum_{n\leq N'} g(\psi_{-1}(dm))(\Lambda-\tfrac{m}{\varphi(m)}\widetilde{\Lambda}_{W_1})(\psi_0(\mathcal{D}mn)) & 1_{n\equiv r_m\pmod q}\ll\|g\|_{\infty} MN w_1^{-1/4}.
\end{aligned}
\end{align}
\end{lemma}

\begin{proof}
We can rewrite the left-hand side of \eqref{eq:Lambdaclaim1} as
\begin{align*}
\sum_{\substack{m\leq M'\\ (m,W_1)=1}} & g(\Delta dm+1)
\sum_{\substack{n\leq N'\\ n\equiv r_m\pmod{q}}}
(\Lambda-\tfrac{m}{\varphi(m)}\widetilde{\Lambda}_{W_1})(\Delta^2\mathcal{D}mn+1)
\\
&=\sum_{\substack{m\leq M'\\ (m,W_1)=1}} g(\Delta dm+1)
\sum_{\substack{\Delta^2\mathcal{D}m+1\leq n'\leq \Delta^2\mathcal{D}mN'+1\\ n'\equiv \Delta^2\mathcal{D}mr_m+1\pmod{\Delta^2\mathcal{D}qm}}}
(\Lambda-\tfrac{m}{\varphi(m)}\widetilde{\Lambda}_{W_1})(n').
\end{align*}
Taking absolute values and expressing each interval as the difference of two initial intervals, the left-hand side of~\eqref{eq:Lambdaclaim1} is
\begin{align}\label{eq:BVgoal}
\leq 2 \|g\|_{\infty}\sum_{\substack{m\leq M\\ (m,W_1)=1}}~\max_{\substack{b\pmod{\Delta^2\mathcal{D}qm}\\ (b,m)=1}}\max_{y\leq \Delta^2\mathcal{D}mN'+1}\left|\sum_{\substack{n'\leq y\\ n'\equiv b\pmod{\Delta^2\mathcal{D}qm}}} (\Lambda-\tfrac{m}{\varphi(m)}\widetilde{\Lambda}_{W_1})(n')\right|.  
\end{align}
We apply Lemma~\ref{le:BV} with $2\Delta^2\mathcal DN$ in place of $N$, $\Delta^2\mathcal Dq$ in place of $q$, $M$ in place of $M'$, and $w_1$ in place of $w$. Recalling \eqref{eq:w1choice}, \eqref{eq:Mchoice}, \eqref{eq:w2choice} and \eqref{eqn_bound_for_cD}, we have
\begin{align*}
\Delta^2\mathcal Dq&\leq w_1,\\
w_1&\leq \exp\left((\log(2\Delta^2\mathcal DN))^{1/3}\right),\\
M&\leq \frac{2\Delta^2\mathcal DN}{(\Delta^2\mathcal Dq)^5},
\end{align*}
so the hypotheses of Lemma~\ref{le:BV} are satisfied. Therefore we obtain the upper bound
\begin{align*}
&\ll \|g\|_\infty M(2\Delta^2\mathcal DN)
\left(\left(\frac{M}{2\Delta^2\mathcal DN}\right)^{1/6}
+w_1^{-1/3}\right)(\log(2\Delta^2\mathcal DN))^{C_0}\\
&\ll\|g\|_\infty MNw_1^{-1/4},
\end{align*}
which proves \eqref{eq:Lambdaclaim1}.
\end{proof}

\section{Proof of third main estimate}
\label{sec_proof_of_sieve_and_weil_bound}

In this section, we establish Propositions~\ref{prop:sieve-bound} and~\ref{prop:weil-bound}.
First, we have that Proposition~\ref{prop:sieve-bound} follows immediately from the following more intricate result (by taking $R=1$ and $H=N$).

\begin{lemma}
\label{lem_singular_series_estimate}
Let $R,m\in\N$ and $d\mid\mathcal D$, and, for $i\in\{1,2\}$, set $W_i^\perp=\frac{W_i}{(R,W_i)}$. Define
\begin{equation}
\label{eqn_g_perp}
g_\perp(n,m,d)= \Lambda_{W_1^\perp}\big(\psi_0(\mathcal{D}mn)\big)
\prod_{j=1}^k \Lambda_{W_2^\perp}\big(\psi_j\big(dm,\tfrac{\mathcal{D}}{d}n\big)\big).
\end{equation}
Then for any $N^{1/10}\leq H\leq N$ we have
\[
\max_{0\leq n\leq N}\Bigg|\frac{1}{H}\sum_{h\leq H} g_\perp(n+hR,m,d) - \mathfrak{S}_R^\perp(m,d)\Bigg| \ll 
\exp\bigl(-2(\log N)^{1/2}\bigr),
\]
where 
\[
\mathfrak{S}_R^\perp(m,d)=\prod_{\substack{p\leq w_1\\ p\nmid R}}\beta_p(m,d)
\]
with $\beta_p(m,d)$ as in \eqref{def_beta_dm}.
\end{lemma}

\begin{proof}
Consider the sets
\[
A_p=
\begin{cases}
\{x\in\Z/p\Z: p\mid \psi_0(\mathcal{D}mx)\prod_{j=1}^k \psi_j(dm,\frac{\mathcal{D}}{d}x)\},&\text{if}~p\leq w_2
\\
\{x\in\Z/p\Z: p\mid \psi_0(\mathcal{D}mx)\},&\text{if}~w_2<p\leq w_1.
\end{cases}
\]
If $A_p=\Z/p\Z$ for some $p\leq w_1$ with $p\nmid R$, then $g_\perp(n,m,d)=0$ for every $n$ and $\beta_p(m,d)=0$. Thus the left-hand side of the required estimate vanishes. We may therefore assume that
\begin{equation}
\label{eqn_Ap_proper}
A_p\neq\Z/p\Z\qquad\text{whenever }p\leq w_1\text{ and }p\nmid R.
\end{equation}
Consider the squarefree integer
\[
B_n\coloneqq
\prod_{\substack{p\leq w_1\\ p\nmid R\\ n\bmod p\in A_p}}p.
\]
Then
\[
g_\perp(n,m,d)
=
C_R\,1_{(W_1,B_n)=1},
\qquad
C_R:=
\frac{W_1^\perp}{\varphi(W_1^\perp)}
\left(\frac{W_2^\perp}{\varphi(W_2^\perp)}\right)^k.
\]

Set $w=w_1$, $W=W_1$ and $Z=H^{1/4}$. 
Applying the fundamental lemma of sieve theory \cite[Fundamental Lemma 6.3]{iw-kow} with sieve level $H^{1/4}$, we can find upper and lower sieve weights $\lambda_d^\pm$ such that
\begin{align*}
\lambda_1^\pm&=1,
\\
|\lambda_d^\pm|&\leq 1,
\\
\lambda_d^\pm&=0
\quad\text{when }d\geq H^{1/4},
\end{align*}
and for every integer $n\geq 1$,
\begin{equation}
\label{eqn_flst_4}
\sum_{d\mid W_1}\lambda_d^-\mathbf 1_{d\mid n}
\leq
\mathbf 1_{(W_1,n)=1}
\leq
\sum_{d\mid W_1}\lambda_d^+\mathbf 1_{d\mid n}.
\end{equation}
Furthermore, for any multiplicative
function $g\colon\mathbb N\to\mathbb R$ satisfying $0\leq g(p)<1$ and for any $K\geq 0$ satisfying
\begin{equation}
\label{eqn_flst_5}
\prod_{u\leq p\leq v}(1-g(p))^{-1}
\leq
\frac{\log v}{\log u}
\left(1+\frac{K}{\log u}\right)
\end{equation}
whenever $2\leq u\leq v\leq w$, we have
\begin{equation}
\label{eqn_flst_6}
\sum_{d\mid W}\lambda_d^\pm g(d)
=
\left(
1+
O\left(
\exp\left(-\frac{\log Z}{\log w}\right)
\left(1+\frac{K}{\log w}\right)^{10}
\right)
\right)
\prod_{p\leq w}(1-g(p)).
\end{equation}

Applying \eqref{eqn_flst_4} to $B_{n+hR}$ and averaging over $h$, we get
\[
\frac{1}{H}\sum_{h\leq H}~\sum_{d\mid W_1}\lambda_d^- 1_{d\mid B_{n+hR}}
\leq
\frac{1}{H}\sum_{h\leq H}\mathbf 1_{(W_1,B_{n+hR})=1}
\leq
\frac{1}{H}\sum_{h\leq H}~\sum_{d\mid W_1}\lambda_d^+ 1_{d\mid B_{n+hR}}.
\]
Note that if $(R,d)>1$ then $d\nmid B_x$, because $B_x$ shares no prime factors with $R$ by definition. Also, $\lambda_d^\pm=0$ for $d\geq H^{1/4}$.
Hence the above upper and lower bound can be refined to
\[
\frac{1}{H}\sum_{h\leq H}~\sum_{\substack{d\mid W_1\\ (d,R)=1 \\ d< H^{1/4}}}\lambda_d^- 1_{d\mid B_{n+hR}}
\leq
\frac{1}{H}\sum_{h\leq H}\mathbf 1_{(W_1,B_{n+hR})=1}
\leq
\frac{1}{H}\sum_{h\leq H}~\sum_{\substack{d\mid W_1\\ (d,R)=1 \\ d< H^{1/4}}}\lambda_d^+ 1_{d\mid B_{n+hR}}.
\]

Next, given some positive integer $d$ with $d\mid W_1$ and $(R,d)=1$, we want to study the expression
\[
\frac{1}{H}\sum_{h\leq H} 1_{d\mid B_{n+hR}}.
\]
We first observe that by the definition of $B_n$ we have
\[
\begin{aligned}
d\mid B_{n+hR} &\iff n+hR\bmod p\in A_p ~\text{for all}~p\mid d
\\
&\iff n+hR\bmod d\in A_d,
\end{aligned}
\]
where $A_d=\{x\in \Z/d\Z: x\bmod p\in A_p~\text{for all}~p\mid d\}$.
Since $(d,R)=1$, the map
$h\mapsto n+hR$ is invertible modulo $d$.
In particular, we have
\[
\#\{1\leq h\leq H: d\mid B_{n+hR}\}
=
H\frac{|A_d|}{d}+\mathrm{O}\bigl(d\bigr).
\]
The Chinese remainder theorem therefore gives, uniformly in $n$, that
\[
|A_d|=\prod_{p\mid d} |A_p|,
\]
which ultimately yields
\[
\frac{1}{H}\sum_{h\leq H} 1_{d\mid B_{n+hR}}= \prod_{p\mid d} \frac{|A_p|}{p} + \mathrm{O}\left(\frac{d}{H}\right).
\]
Let $g$ be the multiplicative function defined by
\[
g(p)=
\begin{cases}
|A_p|/p, &p\leq w_1\text{ and }p\nmid R,\\
0,&\text{otherwise}.
\end{cases}
\]
Noting that
\[
\sum_{d< H^{1/4}} \frac{d}{H}\ll \frac{1}{H^{1/2}},
\]
we arrive at
\[
\sum_{d\mid W_1}\lambda_d^- g(d) + \mathrm{O}\left(\frac{1}{H^{1/2}}\right)
\leq
\frac{1}{H}\sum_{h\leq H}\mathbf 1_{(W_1,B_{n+hR})=1}
\leq
\sum_{d\mid W_1}\lambda_d^+ g(d) + \mathrm{O}\left(\frac{1}{H^{1/2}}\right).
\]

If $p\mid\Delta$, then $A_p=\emptyset$, while if $p\mid R$ or $p>w_1$ then $g(p)=0$. If $p\leq w_1$ and $p\nmid\Delta R$, then~\eqref{eqn_Ap_proper} shows that the polynomial appearing in the definition of $A_p$ is not identically zero modulo $p$. Since its degree is at most $k+1$, we have $|A_p|\leq k+1$. Thus $g(p)\leq (k+1)/p$ for every prime $p$. Now, Mertens' estimate gives
\[
\prod_{u\leq p\leq v}(1-g(p))^{-1}
\ll
\left(\frac{\log v}{\log u}\right)^{k+1},
\]
which implies that the condition \eqref{eqn_flst_5} holds with $K\ll(\log w_1)^{k+1}$.
Then \eqref{eqn_flst_6} gives
\[
\frac{1}{H}\sum_{h\leq H}\mathbf 1_{(W_1,B_{n+hR})=1}=
\prod_{p\leq w_1}(1-g(p))+
O\left(\frac{1}{H^{1/2}}+
\exp\left(-\frac{\log H}{4\log w_1}\right)
(\log w_1)^{10k}
\right).
\]
Since $H\geq N^{1/10}$ and in view of \eqref{eq:w1choice},
\[
\frac{\log H}{4\log w_1}\geq \frac{1}{40} (\log N)^{3/4},
\]
we end up with 
\[
\begin{aligned}
\frac{1}{H}\sum_{h\leq H}\mathbf 1_{(W_1,B_{n+hR})=1}
&=\prod_{p\leq w_1}(1-g(p))+
O\left(\exp\bigl(-3(\log N)^{1/2}\bigr)
\right)
\\
&=\prod_{\substack{p\leq w_1\\ p\nmid R}}\left(1-\frac{|A_p|}{p}\right)+
O\left(\exp\bigl(-3(\log N)^{1/2}\bigr)
\right).
\end{aligned}
\]
Finally, we have
\[
C_R
\prod_{\substack{p\leq w_1\\p\nmid R}}
\left(1-\frac{|A_p|}{p}\right)
=
\prod_{\substack{p\leq w_1\\p\nmid R}}\beta_p(m,d)
=
\mathfrak{S}_R^\perp(m,d).
\]
Since
\[
C_R\ll(\log w_1)^{k+1},
\]
multiplication by $C_R$ leaves an error of size $O(\exp(-2(\log N)^{1/2}))$, and the desired conclusion follows.
\end{proof}

\begin{proof}[Proof of Proposition~\ref{prop:sieve-bound}]
Taking $R=1$, $H=N$, and $n=0$ in Lemma~\ref{lem_singular_series_estimate}, summing over $m$ and $d$, and using
\[
\frac m{\varphi(m)}\Lambda(\psi_{-1}(dm))
\ll(\log N)^2,
\]
Proposition~\ref{prop:sieve-bound} follows.    
\end{proof}

Our next task is to give a proof of Proposition~\ref{prop:weil-bound}, which asserts that
\begin{align}
\label{eqn_3rd_main_estimate}
|S_{\mathrm{main}}^{(2)}-S_{\mathrm{main}}^{(2.5)}|  \ll_A \frac{\tau(\mathcal{D})MN}{(\log N)^A},
\end{align}
where
\begin{align*}
S_{\mathrm{main}}^{(2.5)}&=
\sum_{d\mid \mathcal{D}}\sum_{\substack{m\leq M\\ (m,W_1)=1}}\sum_{n\leq N} \tfrac{m}{\varphi(m)} \Lambda(\psi_{-1}(dm))\Lambda_{W_1}(\psi_0(\mathcal{D}mn))\prod_{j=1}^k \Lambda_{W_2}\big(\psi_j\big(dm,\tfrac{\mathcal{D}}{d}n\big)\big)
\\
S_{\mathrm{main}}^{(2)}&=
\sum_{d\mid \mathcal{D}}\sum_{\substack{m\leq M\\ (m,W_1)=1}}\sum_{n\leq N} \tfrac{m}{\varphi(m)} \Lambda(\psi_{-1}(dm))\widetilde{\Lambda}_{W_1}(\psi_0(\mathcal{D}mn))\prod_{j=1}^k \widetilde{\Lambda}_{W_2}\big(\psi_j\big(dm,\tfrac{\mathcal{D}}{d}n\big)\big).
\end{align*}

For any $\mathbf{t}=(t_0,t_1,\ldots,t_k)\in \{0,1\}^{k+1}$, define
\begin{align*}
z_{\mathbf{t}}(n,m,d)&=\big(\psi_0(\mathcal{D}mn)\big)^{t_0(\widetilde{\beta}_1-1)}
\cdot \prod_{j=1}^k \big(\psi_j\big(dm,\tfrac{\mathcal{D}}{d}n\big)\big)^{t_j(\widetilde{\beta}_2-1)}
\\
f_{\mathbf{t}}(n,m,d)&=
\widetilde{\chi}_1\big(\psi_0(\mathcal{D}mn)\big)^{t_0} \Lambda_{W_1}\big(\psi_0(\mathcal{D}mn)\big)
\\
&\qquad\qquad\cdot 
\prod_{j=1}^k  \widetilde{\chi}_2\big(\psi_j\big(dm,\tfrac{\mathcal{D}}{d}n\big)\big)^{t_j} \Lambda_{W_2}\big(\psi_j\big(dm,\tfrac{\mathcal{D}}{d}n\big)\big)
\end{align*}
It is straightforward to check that
\begin{align*}
S_{\mathrm{main}}^{(2.5)}&= \sum_{d\mid \mathcal{D}} \sum_{\substack{m\leq M\\ (m,W_1)=1}}\tfrac{m}{\varphi(m)}\Lambda(\psi_{-1}(dm))\Bigg(\sum_{n\leq N} f_{\mathbf{0}}(n,m,d)\Bigg),
\\
S_{\mathrm{main}}^{(2)}&= \sum_{\mathbf{t}\in \{0,1\}^{k+1}} 
\sum_{d\mid \mathcal{D}} \sum_{\substack{m\leq M\\ (m,W_1)=1}} \tfrac{m}{\varphi(m)}\Lambda(\psi_{-1}(dm))\Bigg(\sum_{n\leq N} (-1)^{|\mathbf{t}|} z_{\mathbf{t}}(n,m,d) f_{\mathbf{t}}(n,m,d)\Bigg),
\end{align*}
where $\mathbf{0}=(0,0,\ldots,0)$ and $|\mathbf{t}|=t_0+t_1+\ldots+t_k$.
Therefore, instead of \eqref{eqn_3rd_main_estimate}, it suffices to show that for every $\mathbf{t}\neq (0,0,\ldots,0)$ and every $d\mid \mathcal{D}$ we have
\begin{align*}
\sum_{\substack{m\leq M\\ (m,W_1)=1}} \tfrac{m}{\varphi(m)} \bigg|\sum_{n\leq N} z_{\mathbf{t}}(n,m,d) f_{\mathbf{t}}(n,m,d)\bigg| \ll_A \frac{MN}{(\log N)^A}.
\end{align*}
Since $m/\varphi(m)\ll\log\log(3N)$, the preceding estimate follows, after increasing $A$, from
\begin{align}
\label{eqn_3rd_main_estimate_2}
\sum_{\substack{m\leq M\\ (m,W_1)=1}} \bigg|\sum_{n\leq N} z_{\mathbf{t}}(n,m,d) f_{\mathbf{t}}(n,m,d)\bigg| \ll_A \frac{MN}{(\log N)^A}.
\end{align}
Since for each $m$ the function
\[
n\mapsto z_{\mathbf{t}}(n,m,d)
\]
is non-increasing and bounded by $1$, with $|z_{\mathbf t}(n+1,m,d)-z_{\mathbf t}(n,m,d)|\leq(k+1)/n$, using summation by parts and increasing $A$, instead of \eqref{eqn_3rd_main_estimate_2} it is enough to prove uniformly for $N'\leq N$ that 
\begin{align*}
\sum_{\substack{m\leq M\\ (m,W_1)=1}}
\bigg|\sum_{n\leq N'} f_{\mathbf{t}}(n,m,d)\bigg|
\ll_A \frac{MN}{(\log N)^A}.
\end{align*}
Since $Q\mid W_1$, the condition $(m,Q)=1$ is weaker than $(m,W_1)=1$. It therefore suffices to prove
\begin{align*}
\sum_{\substack{m\leq M\\ (m,Q)=1}}
\bigg|\sum_{n\leq N'} f_{\mathbf{t}}(n,m,d)\bigg|
\ll_A \frac{MN}{(\log N)^A}.
\end{align*}
This follows from the following lemma.

\begin{lemma}
\label{lem_cramer_approx_1}
Let $A\geq 1$. For any $d\mid \mathcal{D}$, any $\mathbf{t}=(t_0,t_1,\ldots,t_k)\in\{0,1\}^{k+1}$ with $\mathbf{t}\neq (0,0,\ldots,0)$, and any $N'\leq N$ we have 
\begin{align}
\label{eqn_3rd_main_estimate_3}
\sum_{\substack{m\leq M\\ (m,Q)=1}}\left|\sum_{n\leq N'} f_{\mathbf{t}}(n,m,d)\right|
\ll_A \frac{MN}{(\log N)^A}.
\end{align}
\end{lemma}

\begin{proof}[Proof of Lemma~\ref{lem_cramer_approx_1}]
Let $R_i=(Q,W_i)$ and define $W_i^\perp=\frac{W_i}{R_i}$; note that $R_1=Q$ since $Q\mid W_1$.
Then
\[
\Lambda_{W_i}(n)
=
\Lambda_{R_i}(n)\Lambda_{W_i^\perp}(n),
\]
and hence
\[
f_{\mathbf{t}}(n,m,d)
=g_{\mathbf{t}}(n,m,d)\,g_\perp(n,m,d)
\]
where
\begin{align*}
g_{\mathbf{t}}&= \widetilde{\chi}_1\big(\psi_0(\mathcal{D}mn)\big)^{t_0} \Lambda_{R_1}\big(\psi_0(\mathcal{D}mn)\big)
\prod_{j=1}^k  \widetilde{\chi}_2\big(\psi_j\big(dm,\tfrac{\mathcal{D}}{d}n\big)\big)^{t_j} \Lambda_{R_2}\big(\psi_j\big(dm,\tfrac{\mathcal{D}}{d}n\big)\big)
\\
g_\perp(n,m,d)&= \Lambda_{W_1^\perp}\big(\psi_0(\mathcal{D}mn)\big)
\prod_{j=1}^k \Lambda_{W_2^\perp}\big(\psi_j\big(dm,\tfrac{\mathcal{D}}{d}n\big)\big).
\end{align*}

Since $\qsiegel_1$ divides $\Delta Q$ and $\widetilde{\chi}_1(n)$ is a $\qsiegel_1$-periodic function, we have 
\[
\widetilde{\chi}_1(\psi_0(\mathcal{D}m(n+Q)))=\widetilde{\chi}_1(\psi_0(\mathcal{D}(m+Q)n))=\widetilde{\chi}_1(\psi_0(\mathcal{D}mn)).
\]
In other words, $(n,m)\mapsto \widetilde{\chi}_1(\psi_0(\mathcal{D}mn))$ is $Q$-periodic in both variables.
Similarly, we can show that $(n,m)\mapsto \widetilde{\chi}_2(\psi_j\big(dm,\tfrac{\mathcal{D}}{d}n\big))$ is $Q$-periodic.
Moreover, $R_i\mid Q$, so the factors involving $\Lambda_{R_i}$ are also $Q$-periodic. This shows that $(n,m)\mapsto g_{\mathbf{t}}(n,m,d)$ is a $Q$-periodic function, i.e.,
\[
g_{\mathbf{t}}(n+Q,m,d)=g_{\mathbf{t}}(n,m+Q,d)=g_{\mathbf{t}}(n,m,d).
\]
In particular, for every $h$ we get
\[
\begin{aligned}
\sum_{n\leq N'}g_{\mathbf{t}}(n,m,d)& g_\perp(n,m,d)
\\
&=\sum_{n\leq N'}g_{\mathbf{t}}(n,m,d)g_\perp(n+hQ,m,d)+\mathrm{O}\left(hQ (\log N)^{k+1}\right).
\end{aligned}
\]
Since $Q\leq \exp(2(\log N)^{1/4})$ (see \eqref{eqn:Qbound}), 
by taking $H=\lfloor N^{1/2}\rfloor$ and averaging over all $h\leq H$, this gives
\[
\begin{aligned}
\sum_{n\leq N'}g_{\mathbf{t}}(n,m,d) &  g_\perp(n,m,d)
\\
&=\sum_{n\leq N'}g_{\mathbf{t}}(n,m,d)
\left(\frac1H\sum_{h\leq H}g_\perp(n+hQ,m,d)\right)
+
\mathrm{O}\big(N^{3/4}\big).
\end{aligned}
\]

In light of Lemma~\ref{lem_singular_series_estimate}, applied with $R=Q$, we have uniformly in $n$ that
\[
\frac{1}{H}\sum_{h\leq H}g_\perp(n+hQ,m,d)
=
\mathfrak S_Q^\perp(m,d)
+
O\left(\exp\bigl(-(\log N)^{1/2}\bigr)\right).
\]
Moreover, $\mathfrak S_Q^\perp(m,d)\ll (\log N)^{k+1}$.
Thus, after increasing $A$ by a fixed amount, instead of \eqref{eqn_3rd_main_estimate_3} it suffices to show
\begin{align}
\label{eqn_3rd_main_estimate_5}
\sum_{\substack{m\leq M\\ (m,Q)=1}}\left|\sum_{n\leq N'}
g_{\mathbf{t}}(n,m,d)\right|
\ll_A \frac{MN}{(\log N)^A}.
\end{align}

By $Q$-periodicity, we have
\[
\sum_{\substack{m\leq M\\ (m,Q)=1}}\left|\sum_{n\leq N'}g_{\mathbf{t}}(n,m,d)\right|
=
\frac{MN'}{Q^2}
\sum_{\substack{m\bmod Q\\ (m,Q)=1}}
\left|\sum_{n\bmod Q}g_{\mathbf{t}}(n,m,d)\right|
+
O\left(
(M+N')Q (\log N)^{k+1}
\right).
\]
Using once more $Q\leq \exp(2(\log N)^{1/4})$, the second term is negligible, which reduces \eqref{eqn_3rd_main_estimate_5} to
\begin{align}
\label{eqn_3rd_main_estimate_6}
\sum_{\substack{m\bmod Q\\ (m,Q)=1}}\left|\sum_{n\bmod Q}
g_{\mathbf{t}}(n,m,d)\right|
\ll_A \frac{Q^2}{(\log N)^A}.
\end{align}

If $t_0=1$ and $\widetilde{\chi}_1\equiv0$, or if $t_j=1$ for some $1\leq j\leq k$ and $\widetilde{\chi}_2\equiv0$, then $f_{\mathbf t}\equiv0$ and there is nothing to prove. We may therefore assume that every character selected by $\mathbf t$ is non-zero.

If $t_0=0$, we interpret all local factors involving $\widetilde{\chi}_1$ in the displayed identities below as $1$. Suppose now that $t_0=1$. Since $\qsiegel_1$ has even part either $1$, $4$, or $8$ and squarefree odd part, we can decompose $(\Z/\qsiegel_1\Z)^\times$ as
\[
(\Z/\qsiegel_1\Z)^\times
\cong
(\Z/2^\nu\Z)^\times\times
\prod_{\substack{p\mid \qsiegel_1\\ p~\text{odd}}}
(\Z/p\Z)^\times,
\]
with $\nu\leq 3$.
Hence the primitive quadratic Siegel character $\widetilde{\chi}_1(n)\colon (\Z/\qsiegel_1\Z)^\times\to\{-1,1\}$ splits into a product of local characters
\begin{equation*}
\widetilde{\chi}_1(n)
=
\widetilde{\chi}_{1,2}(n) \cdot 
\prod_{\substack{p\mid \qsiegel_1\\ p~\text{odd}}}
\widetilde{\chi}_{1,p}(n),
\end{equation*}
where $\widetilde{\chi}_{1,2}$ is the $2$-primary component, and for every odd prime $p\mid\qsiegel_1$ the character $\widetilde{\chi}_{1,p}\colon (\Z/p\Z)^\times\to\{1,-1\}$ is a non-principal quadratic character modulo $p$.
In particular this gives
\begin{equation}
\label{eqn_local_decomp_chi_1}
\widetilde{\chi}_1 \big(\psi_0(\mathcal{D}mn)\big)^{t_0}
=
\widetilde{\chi}_{1,2}\big(\psi_0(\mathcal{D}mn)\big)^{t_0} \cdot 
\prod_{\substack{p\mid \qsiegel_1\\ p~\text{odd}}}
\widetilde{\chi}_{1,p}\big(\psi_0(\mathcal{D}mn)\big)^{t_0}.
\end{equation}
We can immediately simplify this expression. Since $\Delta$ is a multiple of $8$, we have $\widetilde{\chi}_{1,2}\big(\psi_0(\mathcal{D}mn)\big)^{t_0}=1$, allowing us to ignore this term.
Also, for $p\mid\Delta$ we have $\widetilde{\chi}_{1,p}\big(\psi_0(\mathcal{D}mn)\big)^{t_0}=1$ and hence these terms can also be omitted from the expression.
Finally, since  for an odd prime $p$, $(\Z/p\Z)^\times$ admits only one non-principal quadratic character, namely the Legendre symbol, which we denote by $\eta_p$, we must have
\[
\widetilde{\chi}_{1,p}=\eta_p.
\] 
Therefore \eqref{eqn_local_decomp_chi_1} simplifies to
\begin{align}
\label{eqn_local_decomp_chi_2}
\widetilde{\chi}_1 \big(\psi_0(\mathcal{D}mn)\big)^{t_0}
= 
\prod_{\substack{p\mid \qsiegel_1\\ p\nmid\Delta}}
\eta_p\big(\psi_0(\mathcal{D}mn)\big)^{t_0}.
\end{align}
By the above convention, \eqref{eqn_local_decomp_chi_2} also holds when $t_0=0$, with both sides equalling $1$.

Using $\Lambda_{R_1}(n)=\prod_{p\mid R_1}\Lambda_p(n)$ and noting that if $p\mid \Delta$ then $\Lambda_p\big(\psi_0(\mathcal{D}mn)\big)=\frac{p}{p-1}$, we can also write
\begin{align*}
\Lambda_{R_1}\big(\psi_0(\mathcal{D}mn)\big)= \Bigg(\prod_{p\mid(\Delta,R_1)} \frac{p}{p-1}\Bigg) \Bigg(\prod_{\substack{p\mid R_1\\ p\nmid\Delta}}\Lambda_p\big(\psi_0(\mathcal{D}mn)\big)\Bigg).
\end{align*}
Using analogous arguments, we can also express $\widetilde{\chi}_2$ and $\Lambda_{R_2}$ in similar product form, which lets us rewrite $g_{\mathbf{t}}(n,m,d)$ as 
\begin{align*}
g_{\mathbf{t}}(n,m,d)= E_0\prod_{\substack{p\mid Q\\ p\nmid\Delta}}
g_{\mathbf{t},p}(n,m,d),
\end{align*}
where $E_0= \prod_{p\mid(\Delta,R_1)} \frac{p}{p-1}\cdot \prod_{p\mid(\Delta,R_2)} \big(\frac{p}{p-1}\big)^k$ is a constant and
\begin{align*}
g_{\mathbf{t},p}(n,m,d)=\eta_p&\big(\psi_0(\mathcal{D}mn)\big)^{t_0 \cdot 1_{p\mid\qsiegel_1}} \Lambda_p\big(\psi_0(\mathcal{D}mn)\big)^{1_{p\mid R_1}}
\\
&\cdot \prod_{j=1}^k \eta_p\big(\psi_j\big(dm,\tfrac{\mathcal{D}}{d}n\big)\big)^{t_j\cdot 1_{p\mid \qsiegel_2}} \Lambda_{p}\big(\psi_j\big(dm,\tfrac{\mathcal{D}}{d}n\big)\big)^{1_{p\mid R_2}}.
\end{align*}
By the Chinese remainder theorem, we now have
\begin{align*}
\sum_{\substack{m\bmod Q\\ (m,Q)=1}}\left|\sum_{n\bmod Q}
g_{\mathbf{t}}(n,m,d)\right|
\ll \prod_{\substack{p\mid Q\\ p\nmid\Delta}}\left(\sum_{\substack{m\bmod p\\ (m,p)=1}}\left| \sum_{n\bmod p}
g_{\mathbf{t},p}(n,m,d)\right|\right).
\end{align*}
This reduces \eqref{eqn_3rd_main_estimate_6} to
\begin{align}
\label{eqn_3rd_main_estimate_7}
\prod_{\substack{p\mid Q\\ p\nmid\Delta}}\left(\sum_{\substack{m\bmod p\\ (m,p)=1}}\left|
 \sum_{n\bmod p}
g_{\mathbf{t},p}(n,m,d)\right|
\right)
\ll_{A} \frac{Q^2}{(\log N)^A}.
\end{align}

We prove \eqref{eqn_3rd_main_estimate_7} separately according to the value of $t_0$.

\smallskip
\noindent\textit{The case $t_0=1$.}

We claim that for all primes $p\mid \qsiegel_1$ with $p\nmid\Delta$ we have
\begin{align}
\label{eqn_3rd_main_estimate_7-2}
\sum_{\substack{m\bmod p\\ (m,p)=1}}\left| \sum_{n\bmod p}
g_{\mathbf{t},p}(n,m,d)\right|\leq C_kp^{3/2}
\end{align}
for some $C_k>0$.

We now prove~\eqref{eqn_3rd_main_estimate_7-2}.
Let $J=\{1\leq j\leq k: t_j=1\}$, and consider
\begin{align*}
g_{\mathbf{t},p}^\bullet(n,m,d)=
\begin{cases}
\eta_p\Big(\big(\psi_0(\mathcal{D}mn)\big) \cdot \prod_{j\in J} \big(\psi_j\big(dm,\tfrac{\mathcal{D}}{d}n\big)\big)\Big),&\text{if}~p\mid \qsiegel_2;
\\
\eta_p\big(\psi_0(\mathcal{D}mn)\big),&\text{otherwise}.
\end{cases}
\end{align*}
We have
\[
\sum_{n\bmod p}
\left| \Lambda_p\big(\psi_0(\mathcal{D}mn)\big)-\frac{p}{p-1}\right|\ll 1
\quad\text{and}\quad
\sum_{n\bmod p}
\left|\Lambda_p\big(\psi_j\big(dm,\tfrac{\mathcal{D}}{d}n\big)\big)-\frac{p}{p-1}\right|\ll 1,
\]
since each affine-linear form occurring here is non-constant modulo $p$ and replacing one factor $\Lambda_p(L(n))$ by $p/(p-1)$ changes the
sum over $n$ by $O_k(1)$, since the difference is supported on the unique zero of $L$. Repeating this for every local von Mangoldt
factor gives
\[
\left|\sum_{n\bmod p}g_{\mathbf t,p}(n,m,d)\right|
\ll\left|\sum_{n\bmod p}g_{\mathbf t,p}^\bullet(n,m,d)\right|+1,
\]
and hence \eqref{eqn_3rd_main_estimate_7-2} follows from
\begin{align}
\label{eqn_3rd_main_estimate_7-3}
\sum_{\substack{m\bmod p\\ (m,p)=1}}\left| \sum_{n\bmod p}
g_{\mathbf{t},p}^\bullet(n,m,d)\right|\ll p^{3/2}.
\end{align}

An important role in the proof of \eqref{eqn_3rd_main_estimate_7-3} is played by the polynomial
\[
P_d(X)
=
\prod_{j=1}^k
\bigl(a_j\Delta^2d^2X^2+\Delta dX-b_j\bigr)
\cdot
\prod_{1\leq i<j\leq k}
\bigl(\Delta dX(a_jb_i-a_ib_j)+b_i-b_j\bigr).
\]
Since $p\mid Q$ and every prime factor of $\mathcal D$ was chosen not to divide $Q$, we have $p\nmid\mathcal D$, and hence $p\nmid d$. We also have $p \nmid \Delta$ by assumption and $p\nmid m$ because $(m,p)=1$. The definition of $\Delta$ gives $p\nmid b_j$ and $p\nmid a_jb_i-a_ib_j$ whenever $i\neq j$. Hence every factor defining $P_d$ is a non-zero polynomial over $\Z/p\Z$, so $P_d$ has $O_k(1)$ roots modulo $p$. 
In particular, we have
\begin{align*}
\sum_{\substack{m\bmod p\\ (m,p)=1}}\left| \sum_{n\bmod p}
g_{\mathbf{t},p}^\bullet(n,m,d)\right|
\ll 
\sum_{\substack{m\bmod p\\ (m,p)=1\\ P_d(m)\not\equiv 0\bmod p}}\left| \sum_{n\bmod p}
g_{\mathbf{t},p}^\bullet(n,m,d)\right|+O_k\left(p\right).
\end{align*}
Consequently, \eqref{eqn_3rd_main_estimate_7-3} follows from
\begin{align}
\label{eqn_3rd_main_estimate_7-4}
\max_{\substack{(m,p)=1\\ P_d(m)\not\equiv 0\bmod p}} \left| \sum_{n\bmod p}
g_{\mathbf{t},p}^\bullet(n,m,d)\right|\ll p^{1/2}.
\end{align}

To show \eqref{eqn_3rd_main_estimate_7-4}, fix $m$ with $(m,p)=1$ and  $P_d(m)\not\equiv 0\bmod p$.
As observed above, $p\nmid\Delta\mathcal D$. Moreover, $p$ is coprime to $m$ because $(m,p)=1$, and the definition of $\Delta$ ensures that $p\nmid b_j$ for every $1\leq j\leq k$. It follows that the affine-linear forms
\[
n\mapsto \psi_0(\mathcal{D}mn)\bmod p
\]
and
\[
n\mapsto \psi_j\big(dm,\tfrac{\mathcal{D}}{d}n\big)\bmod p
\]
are all non-constant.

Two non-constant affine-linear forms have a common root if and only if the corresponding determinant vanishes. A straightforward calculation shows that the determinant of $n\mapsto \psi_0(\mathcal{D}mn)\bmod p$ and $n\mapsto \psi_j\big(dm,\tfrac{\mathcal{D}}{d}n\big)\bmod p$ equals
\[
\Delta^3m^2\mathcal{D}a_jd+\Delta^2\mathcal{D}m-\tfrac{\Delta \mathcal{D}b_j}{d}=\tfrac{\Delta \mathcal{D}}{d} \big(a_j\Delta^2d^2m^2+\Delta d m-b_j\big),
\]
which does not vanish because $\Delta^2m^2d^2a_j+\Delta d m-b_j$ is a factor of the polynomial $P_d(m)$ and we have $P_d(m)\not\equiv 0\bmod p$.

The determinant of $n\mapsto \psi_i\big(dm,\tfrac{\mathcal{D}}{d}n\big)\bmod p$ and $n\mapsto \psi_j\big(dm,\tfrac{\mathcal{D}}{d}n\big)\bmod p$ is
\[
\tfrac{\Delta \mathcal{D}}{d} \big(\Delta dm(a_jb_i-a_ib_j)+b_i-b_j \big)
\]
which is proportional to another factor of  $P_d(m)$ and hence also non-vanishing mod~$p$.

We conclude that the argument of $\eta_p$ in the definition of $g_{\mathbf{t},p}^\bullet$ is a product of non-constant affine-linear forms that have pairwise distinct roots. In particular, it is not a constant multiple of a square in $\mathbb{F}_p[n]$.
By Weil's bound for character sums~\cite[Theorem~11.23]{iw-kow}, we have
\[
\begin{cases}
\left|
\sum_{n\bmod p}
\eta_p\Big(
\big(\psi_0(\mathcal{D}mn)\big)
\prod_{j\in J}
\big(\psi_j\big(dm,\tfrac{\mathcal{D}}{d}n\big)\big)
\Big)
\right|
\ll p^{1/2},&\text{if}~p\mid\qsiegel_2;
\\
\left|
\sum_{n\bmod p}
\eta_p
\big(\psi_0(\mathcal{D}mn)\big)
\right|
\ll p^{1/2},&\text{otherwise}.
\end{cases}
\]
This proves~\eqref{eqn_3rd_main_estimate_7-4}, and hence~\eqref{eqn_3rd_main_estimate_7-2}.

This gives
\[
\prod_{\substack{p\mid \qsiegel_1\\ p\nmid\Delta}}
\left(\sum_{\substack{m\bmod p\\ (m,p)=1}}\left|
 \sum_{n\bmod p}
g_{\mathbf{t},p}(n,m,d)\right|
\right)
\ll_\varepsilon\qsiegel_1^{3/2+\varepsilon},
\]
where we used $C_k^{\omega(\qsiegel_1)}\ll_{k,\varepsilon} \qsiegel_1^{\varepsilon}$.
Since the contribution of primes not dividing $\qsiegel_1$ can be trivially bounded by
\[
\prod_{\substack{p\mid Q\\p\nmid\Delta\\p\nmid  \qsiegel_1}}
\left(\sum_{\substack{m\bmod p\\ (m,p)=1}}\left|
 \sum_{n\bmod p}
g_{\mathbf{t},p}(n,m,d)\right|
\right)
\ll (\log N)^{k+1} \left(\frac{Q}{\qsiegel_1}\right)^2,
\]
we obtain
\begin{align*}
\prod_{\substack{p\mid Q\\ p\nmid\Delta}}\left(\sum_{\substack{m\bmod p\\ (m,p)=1}}\left|
 \sum_{n\bmod p}
g_{\mathbf{t},p}(n,m,d)\right|
\right)
\ll_\varepsilon \frac{Q^2(\log N)^{k+1}}{\qsiegel_1^{1/2-\varepsilon}}.
\end{align*}
Taking $\varepsilon=1/4$ and using \eqref{eq:Siegelbound} proves \eqref{eqn_3rd_main_estimate_7} when $t_0=1$.

\smallskip
\noindent\textit{The case $t_0=0$.}

Let $J=\{1\leq j\leq k\colon t_j=1\}$.
Since $t_0=0$ and $\mathbf t\neq\mathbf0$, the set $J$ is non-empty.
Let $p\mid\qsiegel_2$ with $p\nmid\Delta$, and let $(m,p)=1$. Our choice of $\mathcal D$ gives $p\nmid\mathcal D$, while $p\nmid\Delta$ implies $p\nmid b_j$ for every $j$. Hence all the affine-linear forms occurring in $g_{\mathbf t,p}(n,m,d)$ are non-constant as functions of $n$ modulo $p$.

We replace their local von Mangoldt factors by $p/(p-1)$ one at a time. At each step the difference is supported on the unique root
of one of these forms, so
\[
\sum_{n\bmod p}g_{\mathbf t,p}(n,m,d)
=\left(\frac{p}{p-1}\right)^{k+1}\sum_{n\bmod p}\eta_p\left(\prod_{j\in J}\psi_j\big(dm,\tfrac{\mathcal D}{d}n\big)\right)+O_k(1).
\]
The polynomial $P_d$ is non-zero modulo $p$ and has degree $O_k(1)$, so it has $O_k(1)$ roots modulo $p$. If $P_d(m)\not\equiv0\pmod p$, the affine-linear factors in the character sum are non-constant and have pairwise distinct roots. Their product is therefore not a constant multiple of a square, and Weil's bound gives
\[
\left|\sum_{n\bmod p}g_{\mathbf t,p}(n,m,d)\right| \ll_k p^{1/2}.
\]
For each of the $O_k(1)$ values of $m$ satisfying $P_d(m)\equiv0\pmod p$, the trivial bound is $O_k(p)$. It follows
that
\[
\sum_{\substack{m\bmod p\\(m,p)=1}}\left|\sum_{n\bmod p}g_{\mathbf t,p}(n,m,d)\right|\ll_k p^{3/2}.
\]

For every prime $p\mid Q$ with $p\nmid\Delta$, we also have the trivial estimate
\[
\sum_{\substack{m\bmod p\\(m,p)=1}}\left|\sum_{n\bmod p}g_{\mathbf t,p}(n,m,d)\right|
\leq p^2\left(\frac{p}{p-1}\right)^{k+1}.
\]
Combining these two estimates gives for some $C_k>0$ the bound
\[
\begin{aligned}
\prod_{\substack{p\mid Q\\p\nmid\Delta}} \left(\sum_{\substack{m\bmod p\\(m,p)=1}} \left|\sum_{n\bmod p}g_{\mathbf t,p}(n,m,d)\right|\right)\ll_k C_k^{\omega(\qsiegel_2)}Q^2 \prod_{\substack{p\mid Q\\p\nmid\Delta}} \left(1-\frac{1}{p}\right)^{-k-1} \prod_{\substack{p\mid\qsiegel_2\\p\nmid\Delta}}p^{-1/2}. \end{aligned}
\]
The odd part of $\qsiegel_2$ is squarefree and its $2$-part is at most $8$, so we have
\[
\prod_{\substack{p\mid\qsiegel_2\\p\nmid\Delta}}p \asymp\qsiegel_2.
\]
Mertens' estimate and
$C_k^{\omega(\qsiegel_2)} \ll_{k,\varepsilon}\qsiegel_2^\varepsilon$ 
therefore give
\[
\prod_{\substack{p\mid Q\\p\nmid\Delta}} \left(\sum_{\substack{m\bmod p\\(m,p)=1}}
\left|\sum_{n\bmod p}g_{\mathbf t,p}(n,m,d)\right|\right) \ll_{k,\varepsilon} \frac{Q^2(\log N)^{k+1}}{\qsiegel_2^{1/2-\varepsilon}}.
\]
Taking $\varepsilon=1/4$ and applying \eqref{eq:Siegelbound} with a sufficiently large exponent proves~\eqref{eqn_3rd_main_estimate_7} when $t_0=0$.
\end{proof}

\section{Proof of last main estimate}
\label{sec_proof_last_estimate}

The final missing piece is the proof of Proposition~\ref{prop_last}.

\begin{proof}[Proof of part~\ref{itm_last_i} of Proposition~\ref{prop_last}]
To obtain a lower bound for $\mathfrak{S}(m,d)$, we require lower bounds for $\beta_p(m,d)$ for primes $p\leq w_1$. For sufficiently large $N$, every prime dividing $\Delta\mathcal D$ is at most $w_2$. If $p\mid\Delta$, then all the forms in~\eqref{def_beta_dm} are congruent to $1$ modulo $p$, and hence
\[
\beta_p(m,d)=\left(1-\frac{1}{p}\right)^{-k-1}\geq 1.
\]
If $p\mid\mathcal D$ and $p\nmid\Delta$, then $p\mid d$ and $\psi_0(\mathcal Dmx)\equiv1\pmod p$. Each of the other forms is either congruent to $1$ or is non-constant modulo $p$, so
\[
\beta_p(m,d)\geq \frac{p-k}{p(1-\frac{1}{p})^{k+1}}\gg_k 1.
\]
Since $\Delta$ is fixed and $\mathcal D$ has $\ell$ distinct prime divisors, the product over $p\mid\Delta\mathcal D$ is bounded below uniformly.  Hence it suffices to bound $\beta_p(m,d)$ for primes $p\leq w_1$ with $p\nmid \Delta\mathcal{D}$.
Note that since $p\leq w_1$ and $(m,W_1)=1$, we are always guaranteed that $p\nmid m$.

If $w_2<p\leq w_1$ then 
\[
\# \big\{x\in \mathbb{Z}/p\mathbb{Z}:
p\nmid \psi_0(\mathcal{D}mx) \big\}=p-1,
\]
and hence
\[
\beta_p(m,d)=\frac{p-1}{p(1-\frac{1}{p})}=1.
\]
On the other hand, when $p\leq w_2$, 
\[
\# \big\{x\in \mathbb{Z}/p\mathbb{Z}:
p\nmid \psi_0(\mathcal{D}mx)\prod_{j=1}^k \psi_j\big(dm,\tfrac{\mathcal{D}}{d}x\big)\big\}\geq p-k-1
\]
and so
\[
\beta_p(m,d)\geq \frac{p-k-1}{p(1-\frac{1}{p})^{k+1}}\geq  1-\frac{(k+1)^2}{p^2}.
\]
Note that $p\nmid \Delta$ implies $p>k+1$ (cf.~Definition~\ref{def_Delta}).
This shows that
\[
\mathfrak{S}(m,d)\gg \prod_{k+3<p\leq w_1} \left(1-\frac{(k+1)^2}{p^2}\right)\gg 1,
\]
as desired.
\end{proof}

\begin{proof}[Proof of part~\ref{itm_last_ii} of Proposition~\ref{prop_last}]
Let $\omega(m)$ denote the number of distinct prime factors of $m$. Note that since $m\leq M\leq N$, we have $\omega(m)\ll \log N$.
We can estimate
\begin{align*}
\frac{\varphi(m)}{m} &= \prod_{p\mid m} \left(1-\frac{1}{p}\right)
\\
&\geq \left(1-\frac{1}{w_1}\right)^{\omega(m)}
\\
&\geq \left(1-\frac{1}{w_1}\right)^{\log N}
\\
&\geq 1-\frac{\log N}{w_1},
\end{align*}
where the last inequality follows from Bernoulli's inequality. Since by \eqref{eq:w1choice} we have $\log N/w_1= (\log N) \exp(-(\log N)^{1/4})\ll \exp(-(\log N)^{1/5})$, we are done.
\end{proof}

\begin{proof}[Proof of part~\ref{itm_last_iii} of Proposition~\ref{prop_last}]
Fix $d\mid\mathcal{D}$.
Define 
\begin{equation}
\label{eq:w0choice}
w_0=\exp((\log N)^{1/3})\qquad\text{and}\qquad W_0=\prod_{p\leq w_0} p,
\end{equation}
and let
\[
\widetilde{\Lambda}_{W_0}(m)=
\Lambda_{W_0}(m)\bigl(1-\widetilde{\chi}_0(m)m^{\widetilde{\beta}_0-1}\bigr)
\]
be the Siegel model at level $w_0$ (see Definition~\ref{def:Siegelmodel}).

We make the following two claims:

\begin{claim}
\label{claim1}
We have
\[
\left|\sum_{\substack{m\leq M\\ (m,W_1)=1}}  (\Lambda - \widetilde{\Lambda}_{W_0})(\psi_{-1}(dm))\right|\ll M \exp\bigl(-\frac{1}{1000}(\log N)^{1/12}\bigr).
\]
\end{claim}

\begin{claim}
\label{claim2}
For all $A\geq 1$, 
\[
\left|\sum_{\substack{m\leq M\\ (m,W_1)=1}}  (\Lambda_{W_0}- \widetilde{\Lambda}_{W_0})(\psi_{-1}(dm))\right|\ll_A \frac{M}{(\log N)^A}.
\]
\end{claim}

Before we prove the validity of these claims, let us show how they can be used to finish the proof of \eqref{eqn_itm_last_iii}.
Both claims put together show that
\begin{align*}
\sum_{\substack{m\leq M\\ (m,W_1)=1}}  \Lambda(\psi_{-1}(dm)) &= \sum_{\substack{m\leq M\\ (m,W_1)=1}}  \Lambda_{W_0}(\psi_{-1}(dm)) +\mathrm{O}_A\left(\frac{M}{(\log N)^A}\right)
\\
&= \frac{\varphi(W_1)}{W_1} \sum_{m\leq M}  \Lambda_{W_1}(m) \Lambda_{W_0}(\Delta dm+1) +\mathrm{O}_A\left(\frac{M}{(\log N)^A}\right).
\end{align*}
We apply \cite[Proposition~6.1]{MTW} at scale $M$, with $P(y)=y(\Delta dy+1)$ and
\begin{align*}
f(p,y)=
\begin{cases}
\Lambda_p(y)\Lambda_p(\Delta dy+1),& p\leq w_1,\\
\Lambda_p(\Delta dy+1),& w_1<p\leq w_0,\\
1, &p>w_0.
\end{cases}
\end{align*}
Since $\Delta$ is even, $P$ has no fixed prime divisor. The other hypotheses of that proposition hold with $C=6$ for sufficiently large $N$, since the coefficients of $P$ are $M^{o(1)}$, and we have $w_0\leq \exp((\log M)^{1/2})$, and $|f(p,y)-1|\leq 1_{p\mid P(y)}+6/p$. This gives
\[
\begin{aligned}
\sum_{m\leq M} & \Lambda_{W_1}(m) \Lambda_{W_0}(\Delta dm+1)
\\
&=M \left(\prod_{p\leq w_1} \E_{x\in \Z/p\Z}\Lambda_p(x) \Lambda_p(\Delta dx+1)\right) \left(\prod_{w_1<p\leq w_0} \E_{x\in \Z/p\Z}\Lambda_p(\Delta dx+1) \right)
\\
&\hspace{.55\textwidth}+\mathrm{O}\left(M\exp\bigl(-\tfrac{1}{20}(\log N)^{1/2}\bigr)\right).
\end{aligned}
\]
Since the prime factors of $\Delta$ and $d$ are smaller than $w_1$, we have for all $p$ with $w_1<p\leq w_0$ that
\[
\E_{x\in \Z/p\Z}\Lambda_p(\Delta dx+1)=1.
\]
In particular,
\[
\prod_{w_1<p\leq w_0} \E_{x\in \Z/p\Z}\Lambda_p(\Delta dx+1)=1.
\]
Similarly, if $p\leq w_1$ and $p\nmid \Delta d$ then
\[
\E_{x\in \Z/p\Z}\Lambda_p(x) \Lambda_p(\Delta dx+1)=\frac{p(p-2)}{(p-1)^2}=1+\mathrm{O}\left(\frac{1}{p^2}\right),
\]
and if $p\mid \Delta d$ then
\[
\E_{x\in \Z/p\Z}\Lambda_p(x) \Lambda_p(\Delta dx+1)
=\E_{x\in \Z/p\Z}\Lambda_p(x) \frac{p}{p-1}=\frac{p}{p-1}.
\]
Therefore
\[
\prod_{p\leq w_1} \E_{x\in \Z/p\Z}\Lambda_p(x) \Lambda_p(\Delta dx+1)
\gg 1.
\]
So \eqref{eqn_itm_last_iii} follows as desired.

It remains to prove the two claims.

\begin{proof}[Proof of Claim~1]
It follows from Lemma~\ref{le:Lambdareplace} (applied with $s_1=\frac{1}{10}(\log N)^{1/12}$) that
\[
\sum_{m\leq M} |\Lambda_{W_1}(m)-\Lambda_{W_1}^{(I)}(m)|\ll M e^{-s_1/12},
\]
where $\Lambda_{W_1}^{(I)}(m)$ is defined as
\[
\Lambda_{W_1}^{(I)}(m)=
\frac{W_1}{\varphi(W_1)}\sum_{\substack{r\mid W_1 \\ r\leq w_1^{s_1}}} \mu(r) 1_{r\mid m}.
\]
Therefore
\begin{align*}
\sum_{\substack{m\leq M\\ (m,W_1)=1}}  \Lambda(\psi_{-1}(dm))   
&= \frac{\varphi(W_1)}{W_1} \sum_{m\leq M}  \Lambda_{W_1}(m)\Lambda(\psi_{-1}(dm))
\\
&=\frac{\varphi(W_1)}{W_1} \sum_{m\leq M}  \Lambda_{W_1}^{(I)}(m)\Lambda(\psi_{-1}(dm)) + \mathrm{O}\big(M \exp\bigl(-\frac{1}{1000}(\log N)^{1/12}\bigr)\big)
\\
&= \sum_{\substack{r\mid W_1 \\ r\leq w_1^{s_1}}} \mu(r) \left(\sum_{m\leq M}  1_{r\mid m} \Lambda(\psi_{-1}(dm))\right) + \mathrm{O}\big(M\exp\bigl(-\frac{1}{1000}(\log N)^{1/12}\bigr)\big)
\\
&= \sum_{\substack{r\mid W_1 \\ r\leq w_1^{s_1}}} \mu(r) \left(\sum_{\substack{m\leq \Delta d M+1 \\ m\equiv 1\bmod \Delta dr}}  \Lambda(m)\right) + \mathrm{O}\big(M\exp\bigl(-\frac{1}{1000}(\log N)^{1/12}\bigr)\big).
\end{align*}
Since for any $r\leq w_1^{s_1}$ we have $\Delta d r \leq w_0$, it follows directly from \cite[Proposition 4.5]{MTW} that
\[
\left|\sum_{\substack{1<m\leq \Delta d M+1 \\ m\equiv 1\bmod \Delta dr}}  (\Lambda-\widetilde{\Lambda}_{W_0})(m)\right|\ll M  \exp\bigl(-\tilde{c}(\log N)^{1/2}\bigr)
\]
for some absolute constant $\tilde{c}\in (0,1)$, which yields
\begin{align*}
\sum_{\substack{m\leq M\\ (m,W_1)=1}}  \Lambda(\psi_{-1}(dm))   
= \sum_{\substack{r\mid W_1 \\ r\leq w_1^{s_1}}} & \mu(r) \left(\sum_{\substack{1<m\leq \Delta d M+1 \\ m\equiv 1\bmod \Delta dr}}  \widetilde{\Lambda}_{W_0}(m)\right)+\mathcal{E},
\end{align*}
where $\mathcal{E}$ is an error term satisfying
\begin{align*}
|\mathcal{E}|\ll 
M w_1^{s_1}\exp\bigl(-\tilde{c}(\log N)^{1/2}\bigr)+M\exp\bigl(-\tfrac{1}{1000}(\log N)^{1/12}\bigr)\ll M\exp\bigl(-\tfrac{1}{1000}(\log N)^{1/12}\bigr).    
\end{align*}
Now reversing the Type-I approximation of $\Lambda_{W_1}$, we have
\[
\sum_{\substack{r\mid W_1 \\ r\leq w_1^{s_1}}} \mu(r) \left(\sum_{\substack{1<m\leq \Delta d M+1 \\ m\equiv 1\bmod \Delta dr}}  \widetilde{\Lambda}_{W_0}(m)\right)
=\sum_{\substack{m\leq M\\ (m,W_1)=1}}  \widetilde{\Lambda}_{W_0}(\psi_{-1}(dm)) + \mathrm{O}\big(M\exp\bigl(-\frac{1}{1000}(\log N)^{1/12}\bigr)\big),
\]
finishing the proof of \textit{Claim~1}.
\renewcommand{\qedsymbol}{\ensuremath{\blacksquare}}
\end{proof}

\begin{proof}[Proof of Claim~2] 
If there is no Siegel zero at level $w_0$ then $\Lambda_{W_0}(m)-\widetilde{\Lambda}_{W_0}(m)=0$ and there is nothing to prove. So assume that there is such a zero. Then
\[
\Lambda_{W_0}(m)-\widetilde{\Lambda}_{W_0}(m)=
m^{\widetilde{\beta}_0-1}\widetilde{\chi}_0(m)\Lambda_{W_0}(m)
\]
where $\widetilde{\beta}_0$ denotes this Siegel zero at level $w_0$ and $\widetilde{\chi}_0 \bmod \qsiegel_0$ the corresponding Siegel character and modulus.
Note that in this case we have
\begin{equation}
\label{eqn_q0_bound}
\qsiegel_0\gg_A (\log N)^A    
\end{equation}
by \eqref{eq:Siegelbound}.

Let $s_1=\frac{1}{10}(\log N)^{1/12}$. Arguing as in the proof of \textit{Claim~1}, we have
\begin{align*}
\sum_{\substack{m\leq M\\ (m,W_1)=1}} & (\Lambda_{W_0}-\widetilde{\Lambda}_{W_0})(\psi_{-1}(dm)) 
\\
&=
\sum_{\substack{r\mid W_1 \\ r\leq w_1^{s_1}}}  \mu(r)  \sum_{m'\leq M/r}  (\Lambda_{W_0}-\widetilde{\Lambda}_{W_0})(\psi_{-1}(drm'))+\mathrm{O}\big(M\exp\bigl(-\tfrac{1}{1000}(\log N)^{1/12}\bigr)\big)
\\
&=\sum_{\substack{r\mid W_1 \\ r\leq w_1^{s_1}}}  \mu(r) \sum_{m'\leq M/r} (\psi_{-1}(drm'))^{\widetilde{\beta}_0-1} \widetilde{\chi}_0(\psi_{-1}(drm')) \Lambda_{W_0}(\psi_{-1}(drm'))
\\
&\hspace{0.5\textwidth}+\mathrm{O}\big(M\exp\bigl(-\tfrac{1}{1000}(\log N)^{1/12}\bigr)\big).
\end{align*}
Write $\Sigma_4$ for the double sum after the last equality.

Taking $s_0=1000(\log N)^{1/3}$ and
\[
\Lambda_{W_0}^{(I)}(m)=
\frac{W_0}{\varphi(W_0)}\sum_{\substack{t\mid W_0 \\ t\leq w_0^{s_0}}} \mu(t) 1_{t\mid m},
\]
it follows from Lemma~\ref{le:Lambdareplace} that
\[
\begin{aligned}
\Sigma_4=\frac{W_0}{\varphi(W_0)}\sum_{\substack{r\mid W_1 \\ r\leq w_1^{s_1}}}  \sum_{\substack{t\mid W_0\\ t\leq w_0^{s_0}}} \mu(r) \mu(t)
\sum_{m\leq M/r} (\psi_{-1}(drm))^{\widetilde{\beta}_0-1} \widetilde{\chi}_0(\psi_{-1}(drm)) 1_{t\mid \psi_{-1}(drm)}&
\\
+\mathrm{O}\big(M \exp\bigl(-(\log N)^{1/3} & \bigr)\big).
\end{aligned}
\]
Let $\Sigma_5$ denote the triple sum above.

If $(t,\Delta d r)>1$ then $t\mid \psi_{-1}(drm)$ is impossible. Hence, we can restrict to those $t$ for which $(t,\Delta d r)=1$. Moreover, if $(t,\qsiegel_0)>1$ and $t\mid\psi_{-1}(drm)$, then $\widetilde{\chi}_0(\psi_{-1}(drm))=0$. Hence these terms vanish, and we may further restrict to $(t,\qsiegel_0)=1$.
This leaves us with
\begin{align*}
\Sigma_5&=\sum_{\substack{r\mid W_1 \\ r\leq w_1^{s_1}}}  \sum_{\substack{t\mid W_0\\ t\leq w_0^{s_0}\\ (t,\qsiegel_0\Delta dr)=1}} \mu(r) \mu(t)
\sum_{m\leq M/r} (\psi_{-1}(drm))^{\widetilde{\beta}_0-1} \widetilde{\chi}_0(\psi_{-1}(drm)) 1_{t\mid \psi_{-1}(drm)}.
\end{align*}
For each admissible pair $(r,t)$, let $a_{r,t}\in\{1,\ldots,t\}$ be the unique integer such that
\[
\Delta dr a_{r,t}+1\equiv 0 \pmod t.
\]
Then $t\mid\psi_{-1}(drm)$ if and only if $m\equiv a_{r,t}\pmod t$. Hence
\begin{align*}
\Sigma_5&=\sum_{\substack{r\mid W_1 \\ r\leq w_1^{s_1}}}  \sum_{\substack{t\mid W_0\\ t\leq w_0^{s_0}\\ (t,\qsiegel_0\Delta dr)=1}} \hspace{-.5em}\mu(r) \mu(t)
\sum_{0\leq m\leq M/rt- a_{r,t}/t}\hspace{-.2em}(\psi_{-1}(drtm+dra_{r,t}))^{\widetilde{\beta}_0-1} \widetilde{\chi}_0(\psi_{-1}(drtm+dra_{r,t})).
\end{align*}
By the triangle inequality, we obtain
\begin{align*}
|\Sigma_5|&\leq \sum_{\substack{r\mid W_1\\ r\leq w_1^{s_1} }}  \sum_{\substack{t\mid W_0\\ t\leq w_0^{s_0}\\ (t,\qsiegel_0\Delta dr)=1}} \max_{M'\leq M/rt}
\left|\sum_{0\leq m\leq M'} \widetilde{\chi}_0(\psi_{-1}(drtm+dra_{r,t}))\right|,
\end{align*}
where the last inequality follows from summation by parts.
Letting $W_1'=(\qsiegel_0, W_1)$ and $W_1''=W_1/W_1'$, we have
\[
\begin{aligned}
|\Sigma_5|&\leq
\sum_{u\mid W_1'}~
\sum_{\substack{r\mid W_1''\\ ur\leq w_1^{s_1}}}\sum_{\substack{t\mid W_0\\ t\leq w_0^{s_0}\\ (t,\qsiegel_0\Delta dr)=1}} \max_{M'\leq M/urt}
\left|\sum_{0\leq m\leq M'} \widetilde{\chi}_0(\psi_{-1}(durtm+dura_{ur,t}))\right|
\\
&=
\sum_{u\mid W_1'}~
\sum_{\substack{r\mid W_1''\\ ur\leq w_1^{s_1}}}\sum_{\substack{t\mid W_0\\ t\leq w_0^{s_0}\\ (t,\qsiegel_0\Delta dr)=1}} \max_{M'\leq M/urt}
\left|\sum_{0\leq m\leq M'} \widetilde{\chi}_0(\Delta durtm+b_{u,r,t})\right|,
\end{aligned}
\]
where $b_{u,r,t}=\Delta dura_{ur,t}+1$.

Define
\[
q_u=\frac{\qsiegel_0}{(\Delta d u,\qsiegel_0)}.
\]
The function
\[
m\mapsto\widetilde\chi_0(\Delta durtm+b_{u,r,t})
\]
is $q_u$-periodic. Since $(rt,\qsiegel_0)=1$, we have
\[
(\Delta durt,\qsiegel_0)=(\Delta du,\qsiegel_0).
\]
The odd part of $\qsiegel_0$ is squarefree and its $2$-part divides $\Delta$, so
\[
\big((\Delta du,\qsiegel_0),q_u\big)=1.
\]
Moreover, $b_{u,r,t}\equiv1\pmod{(\Delta du,\qsiegel_0)}$. The Chinese remainder theorem and the primitivity of $\widetilde\chi_0$ therefore give
\[
\sum_{0\leq m<q_u} \widetilde\chi_0(\Delta durtm+b_{u,r,t})=1_{q_u=1}.
\]

Splitting an initial interval into complete periods and a remainder of length less than $q_u$, we obtain
\[
\max_{M'\leq M/urt}\left|\sum_{0\leq m\leq M'}\widetilde\chi_0(\Delta durtm+b_{u,r,t})\right|\ll \frac{M}{urt}1_{q_u=1}+q_u.
\]
Since
\[
1_{q_u=1}\leq\frac{(\Delta du,\qsiegel_0)}{\qsiegel_0}\qquad\text{and}\qquad(\Delta du,\qsiegel_0)\mid\Delta\textnormal{rad}(d)u,
\]
this is
\[
\ll \frac{M\textnormal{rad}(d)}{rt\qsiegel_0}+q_u.
\]

The contribution of the first term is
\[
\ll
\frac{M\textnormal{rad}(d)}{\qsiegel_0}\sum_{u\mid W_1'}1\sum_{r\mid W_1''}\frac{1}{r}\sum_{t\mid W_0}\frac{1}{t}
\ll
\frac{M\textnormal{rad}(d)\tau(\qsiegel_0) \log(2w_1)\log(2w_0)}{\qsiegel_0}.
\]
Since $q_u\leq\qsiegel_0$, the contribution of the second term is
\[
\ll \qsiegel_0\tau(\qsiegel_0)w_1^{s_1}w_0^{s_0}.
\]
Now $\textnormal{rad}(d)\ll(\log N)^{O_k(1)}$,
$\qsiegel_0\leq w_0$, and
\[
w_1^{s_1}w_0^{s_0}=\exp(O((\log N)^{2/3})).
\]
Consequently, \eqref{eqn_q0_bound} gives
\[
|\Sigma_5|\ll_A\frac{M}{(\log N)^A}.
\]
Since $W_0/\varphi(W_0)\ll\log w_0$, Claim~2 follows.
\renewcommand{\qedsymbol}{\ensuremath{\blacksquare}}
\end{proof}
\end{proof}

\begin{proof}[Proof of Proposition~\ref{prop:mainfour}]
All the summands in $S_{\mathrm{main}}^{(3)}$ are non-negative. Restricting to $d$ with $\textnormal{rad}(\mathcal D)\mid d$ and
using parts~\ref{itm_last_i} and~\ref{itm_last_iii} of Proposition~\ref{prop_last} gives
\[
S_{\mathrm{main}}^{(3)}\gg D^\ell N\frac{\varphi(W_1)M}{W_1}.
\]
Since $D^\ell\asymp(D+1)^\ell=\tau(\mathcal D)$, the result follows.
\end{proof}

\bibliography{refs}
\bibliographystyle{plain}

\end{document}